\documentclass{article}

\usepackage{amsmath}
\usepackage{amsfonts}

\usepackage{stmaryrd} 

\usepackage{bussproofs}

\usepackage{tikz-cd}

\usepackage{hyperref}

\usepackage{comment}

\newtheorem{Theorem}{Theorem}

\newtheorem{theorem}[Theorem]{Theorem}
\newtheorem{definition}[Theorem]{Definition}
\newtheorem{lemma}[Theorem]{Lemma}

\newtheorem{corollary}[Theorem]{Corollary}

\newenvironment{proof}{\noindent{\bf Proof:}}{\medskip}

\def\squareforqed{\hbox{\rlap{$\sqcap$}$\sqcup$}}
\def\qed{\ifmmode\squareforqed\else{\unskip\nobreak\hfil
\penalty50\hskip1em\null\nobreak\hfil\squareforqed
\parfillskip=0pt\finalhyphendemerits=0\endgraf}\fi}

\def\cal#1{{\mathcal #1}}

\def\bpi#1{\mathbf{\Pi}^0_{#1}}

\def\lpi#1{{\Pi}^0_{#1}}

\def\sierp{{\mathbb S}}

\def\IF{{\mathbb F}}
\def\IN{{\mathbb N}}

\def\PL{{\mathrm{\bold A}}}
\def\PU{{\mathrm{\bold K}}}
\def\PO{{\mathrm{\bold O}}}

\def\calPfin{{\mathcal P}_{\mathrm{fin}}}

\def\calC{{\cal C}}
\def\calD{{\cal D}}
\def\calE{{\cal E}}
\def\calF{{\cal F}}
\def\calG{{\cal G}}
\def\calH{{\cal H}}
\def\calI{{\cal I}}

\def\I#1{{{\mathrm{\mathbf I}}(#1)}}

\newcommand{\ACA}{\mathsf{ACA}_0}

\newcommand{\ODQPol}{\mathsf{ODS}}
\newcommand{\CHQPol}{\mathsf{CHS}}
\newcommand{\ZCHQPol}{\mathsf{ZCHS}}
\newcommand{\QPol}{\mathsf{QPol}}
\def\QCB{{{\bf QCB}_0}}

\newcommand{\Alg}{\mathsf{Alg}}
\newcommand{\BA}{\mathsf{BA}}

\newcommand{\pt}{\mathsf{pt}}

\newcommand{\Sort}{\mathsf{S}}
\newcommand{\dualSort}{{\mathsf{S}^*}}

\newcommand{\Mor}{\mathsf{Mor}}
\newcommand{\Obj}{\mathsf{Obj}}

\def\src{{\sigma}}
\def\tar{{\tau}}
\def\id{{\iota}}

\newcommand{\uparw}{{\uparrow}}
\newcommand{\dnarw}{{\downarrow}}

\def\int{{\mathrm{\mathbf{int}}}}

\def\wayabovearrow{\rlap{\raise-.25ex\hbox{$\shortuparrow$}}\raise.25ex\hbox{$\shortuparrow$}}
\def\waybelowarrow{\rlap{\raise.25ex\hbox{$\shortdownarrow$}}\raise-.25ex\hbox{$\shortdownarrow$}}

\begin{document}


\title{Effective quasi-Polish categories of overt discrete spaces and compact Hausdorff spaces}

\author{Matthew de Brecht}

\date{\small{Graduate School of Human and Environmental Studies\\ Kyoto University, Japan\\\texttt{matthew@i.h.kyoto-u.ac.jp}}}

\maketitle

\begin{abstract}
We construct the category of overt discrete quasi-Polish spaces and the category of compact Hausdorff quasi-Polish spaces (and some of their subcategories) as internal categories of the category of effective quasi-Polish spaces and computable maps.  

To demonstrate that the constructions are natural, we show that Stone duality is computable, in the sense that the dual contravariant functors and the natural transformations demonstrating their adjointness are computable.
\end{abstract}

\setcounter{tocdepth}{2}
\tableofcontents


\section{Introduction}

In this paper, we construct the category $\ODQPol$ of overt discrete quasi-Polish spaces and the category $\CHQPol$ of compact Hausdorff quasi-Polish spaces\footnote{Classically, compact Hausdorff quasi-Polish spaces are Polish, but the computable points of $\CHQPol$ are not necessarily computable Polish spaces. See the discussion at the end of Section~\ref{sec:CHS}.} as effective quasi-Polish categories (i.e, computable internal categories of $\QPol$, the category of quasi-Polish spaces \cite{dbr}). To demonstrate that the constructions are natural, we show that Stone duality is computable, in the sense that the dual contravariant functors and the natural transformations demonstrating their adjointness are computable. Computability is defined according to the Type Two Theory of Effectivity (TTE) \cite{W00}.

In \cite{dbr21b} we constructed $\QPol$ as a represented space \cite{P16} and showed that products, equalizers, and several powerspace endofunctors on $\QPol$ were computable. Our construction of $\QPol$ itself was clearly not quasi-Polish, in particular the represented space of morphisms was co-analytic complete. $\ODQPol$, $\CHQPol$, and several other categories in this paper provide concrete examples of important subcategories of $\QPol$ that can be represented as internal categories of $\QPol$. 

To formalize Stone duality, we construct the effective quasi-Polish category $\ZCHQPol$ of zero-dimensional compact Hausdorff quasi-Polish spaces as a subcategory of $\CHQPol$, and the effective quasi-Polish categories $\BA(\ODQPol)$ and $\BA(\ZCHQPol)$ of Boolean algebra objects in $\ODQPol$ and $\ZCHQPol$, respectively. The computable Stone duality result is that $\BA(\ODQPol)$ is computably dually equivalent to $\ZCHQPol$, but we also prove the dual Stone duality result that $\BA(\ZCHQPol)$ is computably dually equivalent to $\ODQPol$.

The monads induced by the dual adjunction between $\ODQPol$ and $\ZCHQPol$ determine effective quasi-Polish Eilenberg-Moore categories $\Alg(\ODQPol)$ and $\Alg(\ZCHQPol)$, which we show are computably isomorphic to $\BA(\ODQPol)$ and $\BA(\ZCHQPol)$, respectively. The characterizations as Eilenberg-Moore categories are more suitable for a high level treatment of the duality (see Section~VI.4 of \cite{J82}), but the characterizations as Boolean algebra objects are more suitable for proving the existence of certain morphisms.

Computable versions of Stone duality have been investigated by M.~Harrison-Trainor, A.~Melnikov, and K.M.~Ng in \cite{HMK20} and by N.~Bazhenov, M.~Harrison-Trainor, and A.~Melnikov in \cite{BHM23} (see also the much earlier work by S.~Odintsov and V.~Selivanov \cite{OS89} that precisely characterizes effectively presented Boolean algebras in terms of algebras of clopen subsets of paths through computable trees). Since there are multiple ways of effectivizing the definitions of Boolean algebras and Stone spaces, there are multiple versions of computable Stone duality. The duality in \cite{HMK20} is for computable Boolean algebras (which have decideable equality) and Stone spaces presented by a computable metric and a dense sequence of points. The duality in \cite{BHM23} is for c.e. presented Boolean algebras  (which only have semidecideable equality) and Stone spaces presented by a right-c.e. metric and a dense sequence of points. 

The computable objects of $\BA(\ODQPol)$ are the c.e. presented Boolean algebras (essentially by definition), and the computable objects of $\ZCHQPol$ are computably isomorphic to $\lpi 1$-subspaces of Cantor space (see Lemma~\ref{lem:zdim} in this paper, and compare with Lemma~3 in \cite{OS89}). Therefore, the computable duality between $\BA(\ODQPol)$ and $\ZCHQPol$ in this paper can be viewed as a uniform version of the duality in \cite{BHM23}, although we leave it as an open problem whether it is possible to computably assign (relativized) right-c.e. metrics to all of the objects in $\ZCHQPol$. 

To the best of our knowledge, the computable duality between $\BA(\ZCHQPol)$ and $\ODQPol$ is new. The classical result, however, is well known (see Example (b) of VI-4.6 in \cite{J82}).

The overt-compact and discrete-Hausdorff dualities are studied in detail in P.~Taylor's Abstract Stone Duality (ASD) \cite{Taylor00,Taylor02} and M.~Escard\'{o}'s synthetic topology \cite{ME04}. Our high level treatment of Stone duality closely follows \cite{Taylor00}, in particular our definitions of sobriety and spatiality and our use of a restricted $\lambda$-calculus (see Section~\ref{sec:lambdacalculus}). There are technical differences between the $\lambda$-calculus we introduce and the one in ASD, however, because our motivation for introducing the $\lambda$-calculus is to define computable morphisms in $\ODQPol$ and $\CHQPol$, and so we distinguish between the two categories in our calculus. The $\lambda$-calculus in ASD, on the other hand, assumes a single category of locally compact (or more general) ``spaces''. We leave it as an open problem whether there is an effective quasi-Polish category of locally compact spaces that is a model of ASD. 

Our construction of the quasi-Polish categories $\ODQPol$ and $\BA(\ODQPol)$ are essentially special cases of R.~Chen's construction of quasi-Polish groupoids of countable models for $\sigma$-coherent theories (\cite{Chen19,Chen25}), applied to the empty theory and the theory of Boolean algebras, respectively\footnote{R.~Chen's original construction in \cite{Chen19} included the stronger assumption that equality is decideable (i.e. they are overt discrete \emph{Hausdorff} spaces), but this was relaxed in \cite{Chen25} (see Example~6.9 of that paper). }. The main difference is that our space of morphisms includes all homomorphisms (a category), whereas R.~Chen only includes the isomorphisms (a groupoid). This difference is further investigated in \cite{dbr26b}, where it is shown that the category of models is the ``soberification'' of the groupoid of models in a precise sense. See \cite{dbr26} and \cite{dbr26b} for more applications of $\ODQPol$ in connection to R.~Chen's strong conceptual completeness result for $\sigma$-coherent logic.

The computable objects of $\ODQPol$ are the computably overt computably discrete effective quasi-Polish spaces. These spaces are studied by E.~Neumann, A.~Pauly, C.~Pradic, and M.~Valenti in \cite{NPPV25}, where among other results they characterize them as the spaces of equivalence classes of computably enumerable equivalence relations (ceers; see \cite{GG01}). This is closely related to our construction of $\ODQPol$ using partial equivalence relations (Section~\ref{sec:ODS}).

In \cite{FNPPV25}, J.~Franklin, E.~Neumann, A.~Pauly, C.~Pradic, and M.~Valenti create represented spaces of Polish spaces and compact Polish spaces, where they parameterize spaces using bundles over hyperspaces. They use these constructions to define and study computable categoricity and genericity of compact Polish spaces (in particular, they characterize Cantor space as the unique $\Pi^0_2$-generic compact Polish space). Their representation of compact Polish spaces uses the hyperspace of overt-compact subsets of the Hilbert cube, which has similarities (minus the overtness requirement) to our construction of the space of objects of $\CHQPol$. We leave it as an open problem to clarify the relationship between the two approaches.

It is possible to formalize a large part of the theory of quasi-Polish spaces within second order arithmetic by encoding quasi-Polish spaces as spaces of ideals of countable transitive relations \cite{DPS,dbr20}, and then applying techniques from \cite{mummert:phd} for the reverse mathematics of general topology. We expect that $\ACA$ \cite{Simpson09} should be sufficient for most (if not all) of the results in this paper. In some places we will give hints for formalizing results in $\ACA$, but we will leave a more careful analysis for future work.

The structure of the paper is as follows. After some preliminaries in Section~\ref{sec:preliminaries}, we construct $\ODQPol$ in Section~\ref{sec:ODS}, $\CHQPol$ in Section~\ref{sec:CHS}, and $\ZCHQPol$ in Section~\ref{sec:ZCHS}. We show how to compute finite products and equalizers (uniformly) for these categories in Section~\ref{sec:finlim}. Computable contravariant functors between $\ODQPol$ and $\CHQPol$ are defined in Section~\ref{sec:dualfunctors}, and shown to be computably adjoint. Intuitively, these functors map a space $X$ to the space $2^X$ of all continuous functions from $X$ to the two point discrete space $2$, and they restrict to adjoint functors between $\ODQPol$ and $\ZCHQPol$. Section~\ref{sec:monad} constructs the Eilenberg-Moore categories $\Alg(\ODQPol)$ and $\Alg(\ZCHQPol)$ determined by the induced monads. In Section~\ref{sec:boolalgebra}, we take a different starting point and construct the subcategories of Boolean algebra objects $\BA(\ODQPol)$ and $\BA(\ZCHQPol)$, and prove the existence of enough Boolean algebra homomorphisms to separate points. In Section~\ref{sec:lambdacalculus} we introduce a restricted $\lambda$-calculus for defining and reasoning with morphisms in each category. With the help of the restricted $\lambda$-calculus, we show in Section~\ref{sec:algebras_equiv} that the Eilenberg-Moore categories are computably isomorphic to the categories of Boolean algebra objects. In Section~\ref{sec:StoneDuality}, we use the machinery developed to give relatively simple high-level proofs of the two computable Stone dualities (the usual duality between $\BA(\ODQPol)$ and $\ZCHQPol$, and the second duality between $\BA(\ZCHQPol)$ and $\ODQPol$).

\section{Preliminaries}\label{sec:preliminaries}

We will mainly work within the category of quasi-Polish spaces \cite{dbr}, since it allows straightforward definitions of computability and an easy translation into second order arithmetic \cite{dbr20}. The notion of computability we use coincides with the Type Two Theory of Effectivity (TTE) \cite{W00}. We leave generalizations to larger categories of locales, $\QCB$-spaces \cite{BSS07}, and represented spaces \cite{P16} for future research. 

The Sierpinski space is denoted $\sierp = \{\bot,\top\}$, with $\{\top\}$ open but not closed. We often view (semi-decideable) propositions as having values in Sierpinski space, with $\top$ denoting true and $\bot$ denoting false.

We next give a brief and high level introduction to overtness, compactness, discreteness, and the Hausdorff separation axiom. Our discussion follows the approach of \cite{ME04} and \cite{P16}. Although the definitions we give appear different than the classical definitions, they agree for quasi-Polish spaces.

Given a quasi-Polish space $X$ we equip the spaces $\sierp^X$ and $\sierp^{\sierp^X}$ of continuous functions with the Scott-topologies, which are equivalent to the compact-open topologies since $X$ is quasi-Polish. Note that $\sierp^X$ is quasi-Polish if and only if $X$ is locally compact, but we can treat these spaces as locales within the category of quasi-Polish spaces because their frame $\sierp^{\sierp^X}$ is always quasi-Polish when $X$ is (see \cite{DK19}).

A space $X$ is \emph{(computably) overt} if there is a (computable) continuous function $\exists_X \colon \sierp^X\to\sierp$ such that $\exists_X(U)=\bot \iff U=\emptyset$. $X$ is \emph{ (computably) compact} if there is a (computable) continuous function $\forall_X \colon \sierp^X\to\sierp$ such that $\forall_X(U)=\top \iff U=X$. $X$ is \emph{(computably) discrete} if the ``equals'' function $=\colon X\times X\to\sierp$ is continuous (and computable). $X$ is \emph{(computably) Hausdorff} if the ``not equals'' function $\ne\colon X\times X \to \sierp$ is continuous (and computable).

Given a quasi-Polish space $X$, we write $\PL(X)$ for the lower powerspace of $X$ and $\PU(X)$ for the upper powerspace of $X$ (see \cite{DK19}). The lower powerspace of $X$ is the set of closed (overt) subsets of $X$ with the lower Vietoris topology, which is generated by sets of the form $\Diamond U = \{ A\in\PL(X) \mid A\cap U\not=\emptyset\}$ for $U\in\PO(X)$. $\PL(X)$ is homeomorphic to the subspace of $\sierp^{\sierp^X}$ of all continuous functions $\exists_A\colon \sierp^X\to\sierp$ that preserve finite joins. Given $\exists_A\colon \sierp^X\to\sierp$, the set $\calI = \{ U\in\sierp^X \mid \exists_A(U) = \bot\}$ is a closed ideal in $\sierp^X$ hence the join $V = \bigvee\calI$ is in $\calI$, so the complement $A=X\setminus V$ is the unique closed overt subset of $X$ satisfying $\exists_A(U)=\top \iff U\not\in\calI \iff A\cap U\not=\emptyset$. The upper powerspace of $X$ is the set of saturated compact subsets of $X$ with the upper Vietoris topology, which is generated by sets of the form $\Box U = \{ K\in\PU(X) \mid K\subseteq U\}$ for $U\in\PO(X)$. $\PU(X)$ is homeomorphic to the subspace of $\sierp^{\sierp^X}$ of all continuous functions $\forall_K\colon \sierp^X\to\sierp$ that preserve finite meets. Given $\forall_K\colon \sierp^X\to\sierp$, the set $\calF = \{ U\in\sierp^X \mid \forall_K(U) = \top\}$ is an open filter in $\sierp^X$, so the sobriety of $X$ implies there is a unique saturated compact subset $K$ of $X$ such that $\forall_K(U)=\top\iff U\in\calF \iff K\subseteq U$. 

An alternative and computably equivalent characterization of the lower and upper powerspaces of a quasi-Polish space can be given in terms of spaces of ideals \cite{DPS,dbr20}. First, we recall the space of ideals characterization. Given a transitive relation $\prec$ on $\IN$, a subset $I\subseteq \IN$ is an \emph{ideal} if and only if:
\begin{enumerate}
\item
$I \not=\emptyset$,\hfill (\emph{$I$ is non-empty})
\item
$(\forall a \in I) (\forall b \in \IN)\, (b \prec a \Rightarrow b \in I)$,\hfill (\emph{$I$ is a lower set})
\item
$(\forall a,b \in I)(\exists c\in I)\, (a \prec c  \,\&\,  b \prec c)$.\hfill (\emph{$I$ is directed})
\end{enumerate}
The collection $\I{\prec}$ of all ideals has the topology generated by basic open sets of the form $[a]_{\prec} = \{ I \in \I{\prec} \mid a \in I\}$. It was shown in \cite{DPS} that a space is (effective) quasi-Polish if and only if it is (computably) homeomorphic to a (computably enumerable) transitve relation on $\IN$. 

Given a transitive relation $\prec$ on $\IN$, define the transitive relations $\prec_L$ and $\prec_U$ on $\calPfin(\IN)$ (the set of finite subsets of $\IN$) as:
\begin{eqnarray*}
F \prec_L G &\iff& (\forall m\in F)\,(\exists n\in G)\, m \prec n\\
F \prec_U G &\iff&(\forall n\in G)\,(\exists m\in F)\, m \prec n.
\end{eqnarray*}
It was shown in \cite{dbr20} that $\PL(\I{\prec})$ is (computably) homeomorphic to $\I{\prec_L}$, and $\PU(\I{\prec})$ is (computably) homeomorphic to $\I{\prec_U}$. This correspondence extends to a computable functor on the category of quasi-Polish spaces \cite{dbr21b}. Intuitively, $F\in \calPfin(\IN)$ encodes the basic open set $\bigcap_{n\in F}\Diamond [n]_\prec$ of $\PL(\I{\prec})$ and the basic open set $\Box\bigcup_{n\in F} [n]_\prec$ of $\PU(\I{\prec})$.

In general, if $X$ is compact and $\varphi\colon X\times Y \to \sierp$ is continuous, then taking the transpose $\hat{\varphi}\colon Y\to\sierp^X$ and composing with $\forall_X$ yields a continuous function $\forall_X\circ \hat{\varphi} \colon Y\to\sierp$ that classifies the open subset of all $y\in Y$ satisfying $(\forall x\in X)\,\varphi(x,y)$. As an example, if $X$ is Hausdorff and $K\in\PU(X)$, then composing the transpose of $\ne$ with $\forall_K$ yields the function $\lambda x:X.\forall_K(\lambda y:X. x\ne y)$ from $X$ to $\sierp$ which maps $x$ to $\top$ if and only if $x\not\in K$, hence $K$ is closed.

\subsection{Topological categories}

A \emph{topological category} is an internal category in the category of topological spaces. More concretely, a topological category is a tuple $\calC = (\calC_\Obj, \calC_\Mor, \src,\tar,\id,\circ)$ consisting of the following data:
\begin{itemize}
\item
$\calC_\Obj$ (objects) and $\calC_\Mor$ (morphisms) are topological spaces.
\item
$\src\colon \calC_\Mor\to\calC_\Obj$ (source), $\tar\colon \calC_\Mor\to\calC_\Obj$ (target), and $\id\colon \calC_\Obj\to\calC_\Mor$ (identity) are continuous functions.
\item
$\circ:\subseteq \calC_\Mor\times\calC_\Mor\to\calC_\Mor$ (composition) is a partial continuous function with domain
\[dom(\circ) = \{ \langle g,f \rangle \in \calC_\Mor\times\calC_\Mor \mid \src(g) = \tar(f)\}\]
\end{itemize}
subject to the following:
\begin{itemize}
\item
$\src(g\circ f) = \src(f)$ and $\tar(g\circ f) = \tar(g)$,
\item
$\src(\id(a)) = a$ and $\tar(\id(a)) = a$,
\item
$(h\circ g)\circ f = h \circ (g\circ f)$ when the compositions $h\circ g$ and $g\circ f$ are defined,
\item
if $\src(f) = a$ and $\tar(f) = b$ then $\id(b)\circ f = f = f\circ \id(a)$.
\end{itemize}
See \cite{Chen19} for related work on topological groupoids. See Chapter~7 of \cite{J99} for more on internal categories.

A \emph{quasi-Polish category} is a topological category whose space of morphisms is a quasi-Polish space. Since the space of objects of a topological category is a retract of its space of morphisms, the space of objects of a quasi-Polish category is also a quasi-Polish space. An \emph{effective quasi-Polish category} is a quasi-Polish category whose space of morphisms is an effective quasi-Polish space and the functions $\sigma,\tar,\iota,\circ$ are all computable. This paper will mainly be concerned with effective quasi-Polish categories.

\subsection{Continuous functors and natural transformations}

A \emph{continuous (computable) functor}\footnote{In this paper, ``continuous'' is always meant in the topological sense, not in the categorical sense of preserving limits.} from $\calC= (\calC_\Obj, \calC_\Mor, \src_\calC,\tar_\calC,\id_\calC,\circ_\calC)$ to $\calD= (\calD_\Obj, \calD_\Mor, \src_\calD,\tar_\calD,\id_\calD,\circ_\calD)$ is a pair $F = (F_{\Obj}, F_{\Mor})$ of continuous (computable) functions $F_{\Obj} \colon \calC_\Obj\to\calD_\Obj$ and $F_{\Mor} \colon \calC_\Mor\to\calD_\Mor$ satisfying
\begin{enumerate}
\item
$F_{\Obj}\circ \src_\calC = \src_\calD \circ F_{\Mor}$,
\item
$F_{\Obj}\circ \tar_\calC = \tar_\calD \circ F_{\Mor}$,
\item
$F_{\Mor}\circ \id_\calC = \id_\calD \circ F_{\Obj}$, and
\item
$F_{\Mor}(g\circ_\calC f) = F_{\Mor}(g) \circ_\calD F_{\Mor}(f)$ for all composable $f,g\in\calC_\Mor$.
\end{enumerate}
A continuous functor $F$ from $\calC$ to $\calD$ will be denoted by $F\colon \calC\to\calD$. The identity functor $1_\calC\colon\calC\to\calC$ (which is the identity on objects and morphisms) is clearly computable. Furthermore, if $F\colon\calC\to\calD$ and $G\colon \calD\to\calE$ are (computable) functors, then so is the composition $G\circ F \colon \calC\to\calE$.

Continuous (computable) \emph{contravariant} functors are defined similarly.

Given continuous (computable) functors $F,G\colon \calC\to\calD$, a \emph{continuous (computable) natural transformation} from $F$ to $G$ is given by a continuous (computable) function $\eta\colon \calC_\Obj \to \calD_\Mor$ satisfying:
\begin{itemize}
\item
$\src_{\calD}\circ\eta = F_{\Obj}$ and $\tar_{\calD}\circ \eta = G_{\Obj}$ 

(i.e., $\eta(C) \colon F_{\Obj}(C) \to G_{\Obj}(C)$ for each $C\in\Obj_\calC$), 
\item
 $\eta(\tar_\calC(f)) \circ_\calD F_{\Mor}(f) = G_{\Mor}(f) \circ_\calD \eta(\src_\calC(f))$ for each $f\in\calC_\Mor$.
\end{itemize}
Usually we will write $\eta\colon F\to G$ if it will not cause any confusion.

\section{The category of overt discrete quasi-Polish spaces ($\ODQPol$)}\label{sec:ODS}

This section introduces $\ODQPol$, the effective quasi-Polish category of overt discrete quasi-Polish spaces. See \cite{dbr26} for applications of $\ODQPol$ to computable \'{e}tale spaces.

Let $\PL(\cdot)$ be the lower powerspace monad (see Section~3 of \cite{DK19}).  Recall that a symmetric transitive relation is called a \emph{partial equivalence relation (PER)}. The following definition should be compared with Section~4 of \cite{Chen19}.

\begin{definition}\label{def:ODS}
Define the category $\ODQPol$ as follows:
\begin{itemize}
\item
$\ODQPol_\Obj = \{ \equiv \in \PL(\IN\times\IN)\mid \text{ $\equiv$ is a PER}\}$.
\item
$\ODQPol_\Mor \subseteq \PL(\IN\times \IN) \times \ODQPol_\Obj \times \ODQPol_\Obj$ is the subspace of all tuples $\langle G, \equiv_\src, \equiv_\tar\rangle$  satisfying (for all $a,a',b,b'\in\IN$):
\begin{enumerate}
\item
$G(a,b)$ implies $a\equiv_\src a \,\&\, b\equiv_\tar b$.
\item
$G(a,b) \,\&\, a\equiv_\src a'$ implies $G( a',b)$.
\item
$G(a,b) \,\&\, b\equiv_\tar b'$ implies $G(a,b')$.
\item
$G(a,b) \,\&\, G(a,b')$ implies $b\equiv_\tar b'$.
\item
$a\equiv_\src a$ implies $(\exists b\in\IN)\, G(a,b)$.
\end{enumerate}
\item
$\src(\langle G, \equiv_\src, \equiv_\tar\rangle)= \equiv_\src$
\item
$\tar(\langle G, \equiv_\src, \equiv_\tar\rangle)=\equiv_\tar$
\item
$\id(\equiv)=\langle \equiv, \equiv, \equiv\rangle$.
\item
Composition $\circ$ is defined as
\[\langle G, \equiv_\rho, \equiv_\tar\rangle \circ \langle F, \equiv_\src, \equiv_\rho\rangle = \langle (G\circ F) , \equiv_\src, \equiv_\tar\rangle,\]
where
\[(G\circ F)(a,c) \iff  (\exists b\in \IN)[F(a,b) \,\&\, G(b,c)],\]
which is easily seen to be well-defined and continuous (even computable).
\end{itemize}
\qed
\end{definition}
Technically, $\ODQPol_\Obj$ and $\ODQPol_\Mor$ cannot be constructed as sets in $\ACA$. Instead, using methods similar to \cite{mummert:phd} for formalizing general topology within second-order arithmetic, we can construct within $\ACA$ transitive relations $\prec_{\ODQPol_\Obj}$ and $\prec_{\ODQPol_\Mor}$ on $\IN$ such that $\ODQPol_\Obj$ and $\ODQPol_\Mor$ are (computably) homeomorphic to their respective spaces of ideals. In detail, the overt discrete space $\IN$ of natural numbers is encoded as the space of ideals $\I{=_\IN}$, where $=_\IN$ is the equality relation on the set of natural numbers. Computable constructions for products and the lower powerspace $\PL(\cdot)$ are given in \cite{dbr20,dbr21b}. Computable construction of transitive relations encoding $\lpi 2$ subsets of quasi-Polish spaces was first shown in Theorem~12 of \cite{DPS}, but see also Theorem~3 of \cite{dbr20} for a direct construction that can be done within $\ACA$. C.e. codes (as defined in \cite{dbr20}) for the functions $\src$, $\tar$, $\id$, and $\circ$ can also be constructed in $\ACA$. We obtain the following.

\begin{theorem}
$\ODQPol$ is an effective quasi-Polish category.
\qed
\end{theorem}

Since each point of $\ODQPol_\Mor$ (i.e., each ideal of $\prec_{\ODQPol_\Mor}$) encodes a function between overt discrete spaces, the multiple levels of encoding quickly becomes unmanageable. Therefore, in the following we will treat $\ODQPol_\Obj$ and $\ODQPol_\Mor$ as usual topological spaces, and leave the details of their encodings to the reader. Instead, we will focus on how their points are used to encode spaces and continuous functions. 

The following lemma shows that if we restrict the space of morphisms $\ODQPol_\Mor$ to only contain isomorphisms, then we obtain an effective quasi-Polish groupoid. See \cite{Chen19,Chen25} for more on quasi-Polish groupoids and their applications to logic.

\begin{lemma}\label{lem:odqpol_iso_pi2}
The subspace of $\ODQPol_\Mor$ of isomorphisms is $\lpi 2$.
\end{lemma}
\begin{proof}
Define ${(-)}^{op} \colon \ODQPol_\Mor \to \PL(\IN \times \IN) \times \ODQPol_\Obj \times \ODQPol_\Obj$ as $\langle G,\equiv_\src,\equiv_\tar\rangle \mapsto \langle G^{op},\equiv_\tar,\equiv_\src\rangle$, where $G^{op}(a,b)\iff G(b,a)$. Then $g\in \ODQPol_\Mor$ is an isomorphism if and ony if $g^{op} \in \ODQPol_\Mor$. Since $\ODQPol_\Mor$ is $\lpi 2$, its preimage under ${(-)}^{op}$, the subspace of isomorphisms, is also $\lpi 2$.
\qed
\end{proof}

The next result was first shown in \cite{dbr26}, but we give a slightly modified proof here to emphasize the dual relationship with the category $\CHQPol$ defined in the next section.

\begin{theorem}\label{thrm:ODQPol}
$\ODQPol$ is equivalent to the category of overt discrete quasi-Polish spaces. The computable points of $\ODQPol_\Obj$ are the computably overt computably discrete effective quasi-Polish spaces, and the computable points of $\ODQPol_\Mor$ are the computable functions between computably overt computably discrete effective quasi-Polish spaces.
\end{theorem}
\begin{proof}
We first construct a full and faithful functor $\calF$ from $\ODQPol$ to the category of overt discrete quasi-Polish spaces. 
Given $\equiv$ in $\ODQPol_\Obj$, define $\calF(\equiv)$ to be the subspace of $\PL(\IN)$ of $\equiv$-equivalence classes. Explicitly, 
\[\calF(\equiv) = \{ A\in \PL(\IN) \mid A\not=\emptyset \,\&\, (\forall a\in A)(\forall b\in \IN)[a \equiv b \iff b\in A]\},\]
so it is clear that $\calF(\equiv)$ is a $\bpi 2$-subspace of $\PL(\IN)$ hence quasi-Polish. Next, $\calF(\equiv)$ is discrete because if $A,A'\in \calF(\equiv)$ then $A = A'$ if and only if $A\cap A'\not=\emptyset$ if and only if  $(\exists a\in \IN)[a\in A \,\&\, a\in A']$. Finally, $\calF(\equiv)$ is overt because the function $q_\equiv\colon \IN\to\PL(\IN)$ defined as
$q_\equiv(a)=\{b\in\IN \mid a\equiv b\}$ is continuous (clearly $q_\equiv(a)\in\Diamond U \iff (\exists b\in \IN)[b\in U \wedge (a\equiv b)]$), and $\calF(\equiv)$ is the open subspace of the range of $q_\equiv$ obtained by omitting the closed point $\emptyset\in\PL(\IN)$.

The above constructions are computable, so if $\equiv$ is a computable point then $\calF(\equiv)$ is a computably overt computably discrete effective quasi-Polish space.

For $\langle G, \equiv_\src, \equiv_\tar \rangle$ in $\ODQPol_\Mor$, define $\calF(\langle G, \equiv_\src, \equiv_\tar \rangle)$ to be the function $f_G\colon \calF(\equiv_\src)\to \calF(\equiv_\tar)$ defined as
\[f_G(A) = \{b\in \IN \mid (\exists a\in A)\, G(a,b)\}\]
for each $A\in \calF(\equiv_\src)$. It follows from the five axioms defining elements of  $\ODQPol_\Mor$ that $f_G$ is a well-defined function from $\calF(\equiv_\src)$ to $\calF(\equiv_\tar)$. Finally, $f_G$ is computable from $\langle G, \equiv_\src, \equiv_\tar \rangle$ because
\[f_G(A)\in\Diamond U \iff (\exists b\in \IN)[b\in U \wedge (\exists a\in A)\, G(a,b)].\]
Note that a computable point of $\ODQPol_\Mor$ has computable source and target domains, because $\src$ and $\tar$ are computable.

$\calF$ is easily seen to be a faithful functor, so next we show it is full. Let $\equiv$ and $\equiv'$ be elements of $\ODQPol_\Obj$ and assume $f\colon \calF(\equiv)\to \calF(\equiv')$ is continuous. Let $q_\equiv,q_{\equiv'}\colon \IN\to\PL(\IN)$ be as above. Define $G\in\PL(\IN\times \IN)$ so that $G(a,b)$ holds if and only if $a\equiv a$ and $b\equiv' b$ and $f(q_\equiv(a))=q_{\equiv'}(b)$. Then $\langle G,\equiv,\equiv'\rangle$ is in $\ODQPol_\Mor$ and $\calF(\langle G,\equiv,\equiv'\rangle)=f$.

Next, we show that each overt discrete quasi-Polish space is (computably) homeomorphic to $\calF(\equiv)$ for some $\equiv$ in $\ODQPol_\Obj$. Let $X$ be an overt discrete quasi-Polish space, and fix a transitive relation $\prec$ such that $X \cong \I{\prec}$. 

Let $S = \{ n\in\IN \mid [n]_\prec\not=\emptyset\}$. Since $X$ is discrete, we can enumerate a set of pairs $E\subseteq S\times S$ such that for every $I,J\in \I{\prec}$ we have $I=J$ if and only if there is $\langle a,b\rangle \in E$ with $I\in[a]_\prec$ and $J\in [b]_\prec$. Define
\[S' = \{ c\in S \mid (\exists \langle a,b\rangle\in E)[a\prec c\,\&\, b\prec c]\}.\]  
Note that for each $c\in S'$ there is a unique $I\in\I{\prec}$ such that $\{I\}=[c]_\prec$, and conversely for each $I\in\I{\prec}$ there is at least one $c\in S'$ with $\{I\}=[c]_\prec$. If $[a]_\prec = [b]_\prec = \{I\}$ then $a$ and $b$ must have a $\prec$-upper bound in $I$, so by defining
\[a \equiv_X b \iff a,b\in S' \,\&\, (\exists c\in S)[a\prec c\,\&\, b\prec c],\]
we obtain that $a\equiv_X b$ if and only if $a,b\in S'$ and $[a]_\prec = [b]_\prec=\{I\}$ for a unique $I\in\I{\prec}$. Now define $f\colon \I{\prec}\to \calF(\equiv_X)$ and $g\colon \calF(\equiv_X)\to\I{\prec}$ as
\begin{eqnarray*}
f(I) &=& \{ a\in I \mid a \equiv_X a\}\\
g(A) &=& \{ b\in S\mid (\exists a\in A)\,b\prec a \}.
\end{eqnarray*}
Then $f(I)$ is the set of all $a\in S'$ with $[a]_\prec = \{I\}$, hence $f(I)$ is a $\equiv_X$-equivalence class. Furthermore, $g(A)$ is a $\prec$-ideal because each $\equiv_X$-equivalence class is $\prec$-directed. Therefore, $f$ and $g$ are well-defined, and easily seen to be (computable) continuous inverses of each other. Therefore, $\calF(\equiv_X)$ is (computably) homeomorphic to $X$.
\qed
\end{proof}

\begin{corollary}\label{cor:bij_iso}
If $X$ and $Y$ are overt discrete quasi-Polish spaces, and $f\colon X\to Y$ is a computable bijection, then $f$ is a computable homeomorphism.
\end{corollary}
\begin{proof}
From the previous theorem, we can assume there is $\langle G, \equiv_\src, \equiv_\tar \rangle$ in $\ODQPol_\Mor$ with $\calF(\langle G, \equiv_\src, \equiv_\tar \rangle)=f$. Since $f$ is a bijection, $G$ satisfies
\begin{enumerate}
\item[$4'$.]
$G(a,b) \,\&\, G(a',b)$ implies $a\equiv_\src a'$.
\item[$5'$.]
$b\equiv_\tar b$ implies $(\exists a\in\IN)\, G(a,b)$.
\end{enumerate}
Define $G^{op}(b,a) \iff G(a,b)$. Then $\langle G^{op}, \equiv_\tar, \equiv_\src \rangle$ is in $\ODQPol_\Mor$ and $\calF(\langle G^{op}, \equiv_\tar, \equiv_\src \rangle)$ is the computable inverse of $f$.
\qed
\end{proof}

\section{The category of compact Hausdorff quasi-Polish spaces ($\CHQPol$)}\label{sec:CHS}

Let $\PU(\cdot)$ be the upper powerspace monad (see Section~4 of \cite{DK19}), and let $2^\IN$ be Cantor space.

\begin{definition}
Define the category $\CHQPol$ of compact quasi-Polish Hausdorff spaces as follows:
\begin{itemize}
\item
$\CHQPol_\Obj = \{ \equiv \in \PU(2^\IN\times 2^\IN)\mid \text{ $\equiv$ is a PER}\}$. 
\item
$\CHQPol_\Mor \subseteq \PU(2^\IN \times 2^\IN) \times \CHQPol_\Obj \times \CHQPol_\Obj$ is the subspace of all tuples $\langle G, \equiv_\src, \equiv_\tar \rangle$  satisfying (for all $x,x',y,y'\in 2^\IN$):
\begin{enumerate}
\item
$G( x,y)$ implies $x\equiv_\src x \,\&\, y\equiv_\tar y$
\item
$G(x,y) \,\&\,x\equiv_\src x'$ implies $G( x',y)$
\item
$G(x,y) \,\&\, y\equiv_\tar y'$ implies $G(x,y')$
\item
$G(x,y) \,\&\, G(x,y')$ implies $y\equiv_\tar y'$
\item
$x\equiv_\src x$ implies $(\exists y\in 2^\IN)\, G(x,y)$.
\end{enumerate}
\item
$\src \colon \CHQPol_\Mor\to \CHQPol_\Obj$ is the projection mapping $\langle G, \equiv_\src, \equiv_\tar\rangle$ to $\equiv_\src$.
\item
$\tar \colon \CHQPol_\Mor\to \CHQPol_\Obj$ is the projection mapping $\langle G, \equiv_\src, \equiv_\tar\rangle$ to $\equiv_\tar$.
\item
$\id \colon \CHQPol_\Obj\to \CHQPol_\Mor$ maps $\equiv$ to $\langle \equiv, \equiv, \equiv\rangle$.
\item
Composition $\circ$ is defined as 
\[\langle G, \equiv_\rho, \equiv_\tar\rangle \circ \langle F, \equiv_\src, \equiv_\rho\rangle = \langle (G\circ F) , \equiv_\src, \equiv_\tar\rangle,\]
where
\[(G\circ F)(x,z) \iff  (\exists y\in 2^\IN)[F(x,y) \,\&\, G(y,z)].\]
This is easily seen to be well-defined and continuous (even computable) because $\neg(G\circ F)(x,z) \iff (\forall y\in 2^\IN)[\neg F(x,y) \vee \neg G(y,z)]$, and the latter is universal quantification of an open predicate over a compact space, hence open.
\end{itemize}
\end{definition}

The definition above is dual to Definition~\ref{def:ODS}, and can be formalized in $\ACA$ in essentially the same way by using spaces of ideals (see \cite{dbr20,dbr21b} for computable constructions of the upper powerspace $\PU(\cdot)$).

\begin{lemma}
$\CHQPol$ is an effective quasi-Polish category.
\end{lemma}
\begin{proof}
We can safely leave it to the reader to verify that $\CHQPol$ is a topological category, and that the functions $\sigma,\tar,\iota,\circ$ are all computable.

We first show that $\CHQPol_\Obj$ is effective quasi-Polish. Since the upper powerspaces preserve being effective quasi-Polish \cite{DK19,dbr20}, it suffices to show that the subspace of $\PU(2^\IN\times 2^\IN)$ of symmetric transitive relations is (lightface) $\lpi 2$.

As an example, we show that transitivity is $\lpi 2$, and leave the other proofs to the reader since they can be handled similarly. By definition, a relation $R \in \PU(2^\IN\times 2^\IN)$ is transitive if and only if $(\forall x,y,z \in 2^\IN)[( R(x,y) \,\&\, R(y,z))\Rightarrow R(x,z)]$, which is equivalent to 
\[(\forall x,y,z \in 2^\IN)[( \neg R(x,z) \iff (\neg R(x,z) \wedge (\neg R(x,y)\vee \neg R(y,z)))].\]
Now define $f,g\colon \PU(2^\IN\times 2^\IN)\to \sierp^{ (2^\IN\times 2^\IN \times 2^\IN)}$ as
\begin{eqnarray*}
f(R) &=& \lambda x,y,z:2^\IN.\neg R(x,z)\\
g(R) &=& \lambda x,y,z:2^\IN. \neg R(x,z) \wedge (\neg R(x,y)\vee \neg R(y,z)).
\end{eqnarray*}
Then $f$ and $g$ are computable and $R$ is transitive if and only if $f(R)= g(R)$. Since $\PU(2^\IN\times 2^\IN)$ and $\sierp^{( 2^\IN\times 2^\IN \times 2^\IN)}$ are effective quasi-Polish spaces, it follows that the subspace of transitive relations is $\lpi 2$.

Showing that $\CHQPol_\Mor$ is effective quasi-Polish can be done similarly. We only mention that the 5th condition is equivalent to 
\[(\forall x\in 2^\IN)[ \neg(x\equiv_\src x)\vee (\forall y\in 2^\IN)\,\neg G(x,y)],\]
and $ (\forall y\in 2^\IN)\,\neg G(x,y)$ is an open predicate on $x$, so the whole formula is $\lpi 2$ in the same way as above.
\qed
\end{proof}

The proof of the following is the same as Lemma~\ref{lem:odqpol_iso_pi2}. Again it shows that restricting the space of morphisms  $\CHQPol_\Mor$ to only contain isomorphisms results in an effective quasi-Polish groupoid.

\begin{lemma}\label{lem:chqpol_iso_pi2}
The subspace of $\CHQPol_\Mor$ of isomorphisms is $\lpi 2$.
\qed
\end{lemma}

\begin{theorem}\label{thrm:CHQPol}
$\CHQPol$ is equivalent to the category of compact Hausdorff quasi-Polish spaces. The computable points of $\CHQPol_\Obj$ are the computably compact computably Hausdorff effective quasi-Polish spaces, and the computable points of $\CHQPol_\Mor$ are the computable functions between computably compact computably Hausdorff effective quasi-Polish spaces.
\end{theorem}
\begin{proof}
We first construct a full and faithful functor $\calG$ from $\CHQPol$ to the category of compact Hausdorff quasi-Polish spaces. 

Given $\equiv$ in $\CHQPol_\Obj$, define $\calG(\equiv)$ to be the subspace of $\PU(2^\IN)$ of $\equiv$-equivalence classes. Explicitly, 
\[\calG(\equiv) = \{ K\in \PU(2^\IN) \mid K\not=\emptyset \,\&\, (\forall x\in K)(\forall y\in 2^\IN)[x \equiv y \iff y\in K]\},\]
so it is clear that $\calG(\equiv)$ is a $\bpi 2$-subspace of $\PU(2^\IN)$ hence quasi-Polish. Next, $\calG(\equiv)$ is Hausdorff because if $K,K'\in \calG(\equiv)$ then $K \not= K'$ if and only if $K\cap K'=\emptyset$ if and only if $(\forall x\in 2^\IN)[x\not\in K \vee x\not\in K']$. Finally, $\calG(\equiv)$ is compact because the function $q_\equiv\colon 2^\IN\to\PU(2^\IN)$ defined as
$q_\equiv(x)=\{y\in2^\IN \mid x\equiv y\}$ is continuous (clearly $q_\equiv(x)\in\Box U \iff (\forall y\in 2^\IN)[y\in U \vee \neg (x\equiv y)]$), and $\calG(\equiv)$ is the closed subspace of the range of $q_\equiv$ obtained by omitting the isolated point $\emptyset\in\PU(2^\IN)$. It follows that the restriction of $q_\equiv$ to $\{x\in 2^\IN \mid x\equiv x\}$ is a continuous closed map onto $\calG(\equiv)$, hence a quotient map.

The above constructions are computable, so if $\equiv$ is a computable point then $\calG(\equiv)$ is a computably compact computably Hausdorff effective quasi-Polish space.

For $\langle G, \equiv_\src, \equiv_\tar \rangle$ in $\CHQPol_\Mor$, define $\calG(\langle G, \equiv_\src, \equiv_\tar \rangle)$ to be the function $f_G\colon \calG(\equiv_\src)\to \calG(\equiv_\tar)$ defined as
\[f_G(K) = \{y\in 2^\IN \mid (\exists x\in K)\, G(x,y)\}\]
for each $K\in \calG(\equiv_\src)$. It follows from the five axioms defining elements of  $\CHQPol_\Mor$ that $f_G$ is a well-defined function from $\calG(\equiv_\src)$ to $\calG(\equiv_\tar)$. Finally, $f_G$ is continuous because
\[f_G(K)\in\Box U \iff (\forall y\in 2^\IN)[y\in U \vee (\forall x\in K)\, \neg G(x,y)],\]
and the right hand side only involves universal quantification of open predicates over compact spaces. Clearly, $f_G$ is computable from $\langle G, \equiv_\src, \equiv_\tar \rangle$. Note that a computable point of $\CHQPol_\Mor$ has computable source and target domains, because $\src$ and $\tar$ are computable.

	$\calG$ is easily seen to be a faithful functor, so next we show it is full. Let $\equiv$ and $\equiv'$ be elements of $\CHQPol_\Obj$ and assume $f\colon \calG(\equiv)\to \calG(\equiv')$ is continuous. Let $q_\equiv,q_{\equiv'}\colon 2^\IN\to\PU(2^\IN)$ be as above. Define $G\in\PU(2^\IN\times 2^\IN)$ so that $G(x,y)$ holds if and only if $x\equiv x$ and $y\equiv' y$ and $f(q_\equiv(x))=q_{\equiv'}(y)$. Then $\langle G,\equiv,\equiv'\rangle$ is in $\CHQPol_\Mor$ and $\calG(\langle G,\equiv,\equiv'\rangle)=f$.

Next, we show that each compact Hausdorff quasi-Polish space is (computably) homeomorphic to $\calG(\equiv)$ for some $\equiv$ in $\CHQPol_\Obj$. Let $X$ be a compact Hausdorff quasi-Polish space, and fix a transitive relation $\prec$ such that $X \cong \I{\prec}$. Then $X \in \PU(X) \cong \I{\prec_U}$, so we can enumerate a sequence
\[F_0 \prec_U F_1 \prec_U \cdots\]
such that for each finite $F\subseteq\IN$, $X \subseteq \bigcup_{n\in F} [n]_\prec$ if and only if $(\exists i\in\IN)\, F \prec_U F_i$. Intuitively, $(F_i)_{i\in\IN}$ is an enumeration of finite coverings of $X$ that get arbitrarily fine. Let $s_i$ be the cardinality of $F_i$, and for $n<s_i$ let $F_i(n)$ be the $n^{th}$ element of $F_i$ under some fixed ordering. For $x,y \in 2^\IN$ define $(x\equiv_X y)$ if and only if
\begin{itemize}
\item
$x = 0^{m_0} 1 0^{m_1} 1 0^{m_2} 1 \cdots$ and $y = 0^{n_0} 1 0^{n_1} 1 0^{n_2} 1 \cdots$ where $m_i, n_i < s_i$  for each $i\in\IN$,
\item
$F_i(m_i) \prec F_{i+1}(m_{i+1})$ and $F_i(n_i) \prec F_{i+1}(n_{i+1})$ for each $i\in\IN$, and
\item
$[F_i(m_i)]_\prec \cap [F_i(n_i)]_\prec \not=\emptyset$ for each $i\in\IN$.
\end{itemize}
The first two requirements guarantee that $x$ and $y$ are encoding $\prec$-increasing sequences, and the last requirement along with the Hausdorff assumption guarantees that $x$ and $y$ generate the same $\prec$-ideal. We leave it to the reader to check that $\equiv_X$ is in $\CHQPol_\Obj$, since it is similar to our previous arguments. However, we point out that for the computable case, if  $X \cong \I{\prec}$ for a c.e. relation $\prec$, then by Proposition~1 of \cite{DKS24} we can computably obtain a computable relation $\sqsubset$ such that $X$ is computably homeomorphic to $\I{\sqsubset}$, so without loss of generality we can assume the tests in the second item defining $\equiv_X$ are decideable. Also, if $x$ and $y$ are encoding different ideals, then being computably Hausdorff implies $[F_i(m_i)]_\prec \cap [F_i(n_i)]_\prec =\emptyset$ will eventually be observed for sufficiently large $i\in\IN$.

Finally, we show that $X\cong \I{\prec}$ is (computably) homeomorphic to $\calG(\equiv_X)$. Define $f\colon \I{\prec} \to \calG(\equiv_X)$ as follows. For $I\in \I{\prec}$, define $x\in f(I)$ if and only if $x \equiv_X x$ and the $\prec$-increasing sequence encoded by $x$ generates $I$. Then $f(I) \in \Box U$ if and only if every $x$ encoding a $\prec$-increasing sequence that generates $I$ is in $U$. Since $\I{\prec}\cong X$ is Hausdorff, the subspace $V=\{J\in \I{\prec} \mid J\not= I\}$ is open. Therefore, $f(I)\in \Box U$ if and only if for every $x$ encoding a $\prec$-increasing sequence, either $x\in U$ or the ideal generated by $x$ is in $V$. Since the $x$ encoding a $\prec$-increasing sequence correspond to the $\prec$-ascending paths through $(F_i)_{i\in\IN}$, and these form a finitely branching tree, the proposition $f(I) \in \Box U$ is semi-decidable given $I$, $U$, and the compact Hausdorff structure of $\I{\prec}\cong X$.

Next, define $g\colon \calG(\equiv_X) \to \I{\prec}$ as $g(K) = I$ if and only if there is $x\in K$ such that the $\prec$-increasing sequence encoded by $x$ generates $I$. Then $g(K) \in [n]_\prec$ if and only if for every $x\in K$ (which is necessarily of the form $x=0^{m_0} 1 0^{m_1} 1 0^{m_2} 1 \cdots$)  there is $i\in\IN$ such that $n\prec m_i$. The computability of $g$ immediately follows from the compactness of $K$. The proof that $f$ and $g$ are inverses to each other can be left to the reader.
\qed
\end{proof}

We obtain the following corollary, whose proof is the same as Corollary~\ref{cor:bij_iso}.

\begin{corollary}\label{cor:bij_iso_ch}
If $X$ and $Y$ are compact Hausdorff quasi-Polish spaces, and $f\colon X\to Y$ is a computable bijection, then $f$ is a computable homeomorphism.
\qed
\end{corollary}

Classically, every compact Hausdorff space is regular, and every regular quasi-Polish space is Polish. Therefore, $\CHQPol$ is classically equivalent to the category of compact Polish spaces. Furthermore, by Proposition~3.6 of \cite{AH} and Theorem~6.1 of \cite{Sch98}, each computable point of $\CHQPol_\Obj$ admits a computable metric. However, the computable points of $\CHQPol_\Obj$ are not necessarily computably overt, hence they are not always computable Polish spaces in the sense of \cite{HMK20} and \cite{BHM23}. We also leave as an open problem whether it is possible to computably assign metrics to all of the objects of $\CHQPol$.

\section{The category of zero-dimensional compact Hausdorff quasi-Polish spaces ($\ZCHQPol$)}\label{sec:ZCHS}

\begin{lemma}\label{lem:zdim}
Let $\equiv$ be an object of $\CHQPol_\Obj$. The following are equivalent:
\begin{itemize}
\item[(i)]
For all $x,y\in2^\IN$,
\[x\equiv y \iff (\forall a,b\in \IN)[a \equiv' b \Rightarrow \Phi_a(x) = \Phi_b(y)],\]
where $\equiv'$ are PERs corresponding to $Y$ and $\calE(Y)$, respectively.
\item[(ii)]
$\calG(\equiv)$ is computably isomorphic to a closed subspace of $2^\IN$, where $\calG$ is the functor from the proof of Theorem~\ref{thrm:CHQPol}.
\end{itemize}
Furthermore, if $\equiv$ is a computable element of $\CHQPol_\Obj$, then the above are also equivalent to:
\begin{itemize}
\item[(iii)]
$\calG(\equiv)$ is computably isomorphic to a $\lpi 1$-subspace of $2^\IN$.
\end{itemize}
\end{lemma}
\begin{proof}
Let $Y = \calG(\equiv)$. To show (i) implies (ii), fix an enumeration $(a_i)_{i\in\IN}$ of $\{ a\in\IN \mid a\equiv' a\}$, and define $f\colon Y \to 2^\IN$ as $f(y)(i) = \Phi_{a_i}(y)$. Then (i) implies $f$ is a computable bijection onto its image, which is closed because $Y$ is compact. It follows from Corollary~\ref{cor:bij_iso_ch} that $Y$ is computably isomorphic to its image under $f$.

If $Y$ is a computable element of $\CHQPol_\Obj$, then (i) implies (iii) because the image of $Y$ under the $f$ just constructed is a $\lpi 1$-subspace of $2^\IN$. The implication from (iii) to (ii) is trivial.

To see that (ii) implies (i), assume $f\colon \calG(\equiv) \to 2^\IN$ is an embedding. Let $q_\equiv \colon 2^\IN \to \PU(2^\IN)$ be as in the proof of Theorem~\ref{thrm:CHQPol}. The $(\Rightarrow)$ direction of (i) holds by definition, so we only need to show the $(\Leftarrow)$ direction. Assume $\neg(x\equiv y)$. 

First consider the case when $\neg(x\equiv x)$. Then  $U = \{ z\in 2^\IN \mid q_{\equiv}(z)=\emptyset\}$ is an open neighborhood of $x$, so there is a clopen $A\subseteq 2^\IN$ such that $x\in A \subseteq U$. Let $a,b\in\IN$ be such that $\Phi_a$ is equal to $1$ everywhere, and $\Phi_b$ is equal to $1$ on $A$ and equal to $0$ on the complement of $A$. Then $a \equiv' b$ but $\Phi_a(x)\not=\Phi_b(y)$. Therefore, (i) holds. The case when $\neg(y\equiv y)$ can be handled similarly.

Next consider the case when $x \equiv x$ and $y\equiv y$. Then $f(q_{\equiv}(x))$ and $f(q_{\equiv}(y))$ are both defined and distinct. Let $B\subseteq 2^\IN$ be a clopen neighborhood of $f(q_{\equiv}(x))$ that does not contain $f(q_{\equiv}(y))$. Let $C = \{ z \in 2^\IN \mid (z \equiv z) \wedge   f(q_{\equiv}(z)) \in B\}$ and let $C' = \{ z \in 2^\IN \mid (z \equiv z) \wedge   f(q_{\equiv}(z)) \not\in B\}$. Then $C$ and $C'$ are disjoint closed subsets of $2^\IN$, so there exists clopen $A\subseteq 2^\IN$ such that $C\subseteq A$ and $C'\cap A=\emptyset$. Let $a\in\IN$ be such that $\Phi_a$ is equal to $1$ on $A$ and equal to $0$ on the complement of $A$. Then $a \equiv' a$ but $\Phi_a(x)\not=\Phi_a(y)$. Therefore, (i) holds.

In either case, we see that (i) holds, hence (ii) implies (i).
\qed
\end{proof}

The objects of $\CHQPol$ satisfying the equivalent conditions of the above lemma are called \emph{zero-dimensional}. We write $\ZCHQPol$ for the full subcategory of $\CHQPol$ of zero-dimensional spaces. 

\begin{lemma}
$\ZCHQPol$ is an effective quasi-Polish category.
\end{lemma}
\begin{proof}
$\ZCHQPol_\Obj$ is the subspace of $\CHQPol_\Obj$ of all PERs $\equiv$ satisfying 
\[(\forall x,y\in 2^\IN)\big[\neg(x\equiv y) \iff (\exists a,b\in \IN)[a \equiv' b \,\&\, \Phi_a(x) \not= \Phi_b(y)]\big],\]
which is $\lpi 2$ because $\lpi2 $ sets are closed under universal quantification over $2^\IN$. Explicitly, define $f,g\colon \CHQPol_\Obj\to \sierp^{ (2^\IN\times 2^\IN)}$ as
\begin{eqnarray*}
f(\equiv) &=& \lambda x,y:2^\IN.\neg (x\equiv y)\\
g(\equiv) &=& \lambda x,y:2^\IN. (\exists a,b\in \IN)[a \equiv' b \,\&\, \Phi_a(x) \not= \Phi_b(y)].
\end{eqnarray*}
Then $f$ and $g$ are computable and $\equiv$ is zero-dimensional if and only if $f(\equiv)= g(\equiv)$, hence $\ZCHQPol_\Obj$ is $\lpi 2$.

Finally, since $\ZCHQPol_\Mor$ is the subspace of $\CHQPol_\Mor$ of morphisms whose source and target are zero-dimensional, it is equal to $\src^{-1}(\ZCHQPol_\Obj)\cap \tar^{-1}(\ZCHQPol_\Obj)$, which is $\lpi 2$ by the computability of $\src$ and $\tar$.
\qed
\end{proof}

\section{Finite products and equalizers}\label{sec:finlim}


In this section, we show how to compute finite products and equalizers within $\ODQPol$, $\CHQPol$, and $\ZCHQPol$.

\begin{lemma}
The terminal object of $\ODQPol$ is given by the PER $\IN\times\IN$ (i.e., all elements are equivalent to each other), and the terminal objects of $\CHQPol$ and $\ZCHQPol$ are given by the PER $2^\IN \times 2^\IN$.
\end{lemma}
\begin{proof}
If $\equiv$ is another object, then the unique map into the terminal object is given by $G(a,b) \iff a\equiv a$ (for any $b$).
\qed
\end{proof}

Let $\langle \cdot,\cdot\rangle \colon \IN\times \IN \to \IN$ be the Cantor pairing function (i.e., $\langle a,b\rangle = \frac{1}{2}(a+b)(a+b+1)+b$), and let $\langle \cdot,\cdot \rangle \colon 2^\IN\times 2^\IN \to 2^\IN$ be the computable bijection that interleaves two sequences (i.e., $\langle x,y\rangle(2a) = x(a)$ and $\langle x,y\rangle(2a+1) = y(a)$).

\begin{lemma}
The product of two objects $\equiv_1$ and $\equiv_2$ in $\ODQPol$ is given by the PER $\equiv_\times$ defined as
\[\langle a_1,a_2\rangle \equiv_\times \langle b_1,b_2\rangle \iff a_1\equiv_1 b_1 \,\&\, a_2 \equiv_2 a_2\]
along with the projections 
\begin{eqnarray*}
\pi_1 &=& \langle P_1,\equiv_\times,\equiv_1\rangle,  \text{ where }P_1(\langle a_1,a_2\rangle, a) \iff a_1 \equiv_1 a \,\&\, a_2 \equiv_2 a_2, \text{ and }\\
\pi_2 &=& \langle P_2,\equiv_\times,\equiv_2\rangle,  \text{ where }P_2(\langle a_1,a_2\rangle, a) \iff a_2 \equiv_2 a \,\&\, a_1 \equiv_1 a_1.
\end{eqnarray*}
Products in $\CHQPol$ and $\ZCHQPol$ are defined identically.
\end{lemma}
\begin{proof}
Given morphisms $f=\langle F, \equiv, \equiv_1\rangle$ and $g=\langle G, \equiv, \equiv_1\rangle$, the morphism $u=\langle U,\equiv,\equiv_\times\rangle$, defined as 
\[U(a,\langle a_1,a_2\rangle) \iff F(a,a_1) \,\&\, G(a,a_2),\]
is the unique morphism satisfying $f = \pi_1\circ u$ and $g=\pi_2\circ u$.
\qed
\end{proof}

\begin{lemma}
Given morphisms $f=\langle F,\equiv_\src,\equiv_\tar\rangle$ and $g=\langle G,\equiv_\src,\equiv_\tar\rangle$ in $\ODQPol$, their equalizer is given by the PER $\equiv_E$ defined as
\[a \equiv_E a' \iff a\equiv_\src a' \,\&\, (\exists b\in \IN)[F(a,b) \,\&\, G(a',b)],\]
along with the morphism $e = \langle E,\equiv_E,\equiv_\src\rangle$ defined as $E(a,a') \iff a\equiv_E a'$. Equalizers in $\CHQPol$ and $\ZCHQPol$ are defined similarly.
\end{lemma}
\begin{proof}
Given a morphism $h=\langle H,\equiv,\equiv_\src\rangle$ satisfying $f\circ h = g\circ h$, the morphism $u = \langle U,\equiv,\equiv_E\rangle$, defined as $U(a,b) \iff H(a,b)$, is the unique morphism satisfying $h = e\circ u$.

The same definition of $\equiv_E$ (but replacing $\IN$ with $2^\IN$) also works for $\CHQPol$ and $\ZCHQPol$.
\qed
\end{proof}

Note that equalizers can be computed uniformly in the following sense.

\begin{theorem}\label{thrm:comp_equalizers}
Let $X$ be an effective quasi-Polish space, and $\calC$ be either $\ODQPol$ or $\CHQPol$ or $\ZCHQPol$. Let $f,g\colon X \to \calC_\Mor$ be computable functions such that $\src\circ f = \src \circ g$ and $\tar\circ f=\tar\circ g$. Then there is a computable function $e\colon X\to \calC_\Mor$ such that $e(x)$ is the equalizer of $f(x)$ and $g(x)$ for each $x\in X$.
\qed
\end{theorem}


\section{Dual contravariant functors}\label{sec:dualfunctors}


We next construct (computable) continuous contravariant functors $\calD\colon \ODQPol\to \CHQPol$ and $\calE\colon \CHQPol \to \ODQPol$. The intended interpretation in both cases is the contravariant functor that maps a space $X$ to the space $2^X$ of continuous functions, where $2=\{0,1\}$ is the usual discrete space and $2^X$ has the compact-open topology. A continuous function $f\colon X\to Y$ is mapped to $2^f\colon 2^Y\to 2^X$ defined as $2^f(h) = h\circ f$ for each $h \in 2^Y$.

\begin{definition}\label{def:functor_D}
Define the contravariant functor $\calD\colon \ODQPol\to \CHQPol$ as follows:
\begin{itemize}
\item
For $\equiv$ in $\ODQPol_\Obj$, define $\calD_\Obj(\equiv)$ to be the object $\equiv'$ in $\CHQPol_\Obj$ defined as
\[x\equiv' y \iff (\forall a,b\in\IN)[a\equiv b \Rightarrow x(a)  = y(b)].\]
\item
For $\langle G, \equiv_\src, \equiv_\tar\rangle$ in $\ODQPol_\Mor$, define $\calD_\Mor(\langle G, \equiv_\src, \equiv_\tar\rangle)$ to be the morphism $\langle G', \equiv'_\tar, \equiv'_\src\rangle$ in $\CHQPol_\Mor$ defined as:
\begin{itemize}
\item
$\equiv'_\src = \calD_\Obj(\equiv_\src)$, 
\item
$\equiv'_\tar = \calD_\Obj(\equiv_\tar)$,  and
\item
$G'\in\PU(2^\IN\times 2^\IN)$ is defined so that $G'(y,x)$ holds if and only if 
\[y \equiv'_\tar y \,\&\, x \equiv'_\src x \,\&\, (\forall a,b\in\IN)[G(a,b) \Rightarrow x(a) = y(b)]. \]
\end{itemize}
\end{itemize}
\qed
\end{definition}

It is straightforward to show that $\calD=(\calD_\Obj,\calD_\Mor)$ is a well-defined continuous (computable) contravariant functor.

We briefly verify that the above definition behaves as we intended. First assume $\equiv$ is in $\ODQPol_\Obj$ and corresponds to the overt discrete space $\calF(\equiv)$ via the equivalence of Theorem~\ref{thrm:ODQPol}. Then $\equiv' = \calD_\Obj(\equiv)$ is defined so that $x,y\colon \IN\to 2$ are $\equiv'$-equivalent if and only if they are both $\equiv$-invariant and agree on the equivalence classes of $\equiv$. Therefore, the equivalence classes of $\equiv'$ are in bijection with the set of functions from $\calF(\equiv)$ to $2$. Now let $\calG(\equiv')$ be the compact Hausdorff space corresponding to $\equiv'$ via Theorem~\ref{thrm:CHQPol}. We show the bijection from $\calG(\equiv')$ to $2^{\calF(\equiv)}$ (with the compact-open topology) is continuous. Let $K\subseteq \calF(\equiv)$ be compact and $U\subseteq 2$ open, and let $[K,U] \subseteq 2^{\calF(\equiv)}$ be the set of functions that map $K$ into $U$. $K$ is finite because $\calF(\equiv)$ is discrete, so we obtain a finite set $F\subseteq \IN$ by choosing one element from each equivalence class in $K$. Let $V = \{ x\in 2^\IN \mid \neg (x\equiv x) \vee (\forall a\in F)\, x(a)\in U\}$. Then $\Box V$ is the open subset of $\calG(\equiv')$ that gets mapped by the above bijection into $[K,U]\subseteq 2^{\calF(\equiv)}$. Since $\calG(\equiv')$ and $2^{\calF(\equiv)}$ are compact Hausdorff, it follows that they are homeomorphic.

Next, assume $\langle G, \equiv_\src, \equiv_\tar\rangle$ is in $\ODQPol_\Mor$ and corresponds to the function $f\colon \calF(\equiv_\src)\to \calF(\equiv_\tar)$. This is mapped by $\calD_\Mor$ to a morphism $\langle G', \equiv'_\tar, \equiv'_\src\rangle$ corresponding to a continuous function $2^f \colon 2^{\calF(\equiv_\tar)} \to 2^{\calF(\equiv_\src)}$. From the definition of $G'$, we have $2^f(y) = x$ if and only if $f(a) = b\Rightarrow x(a)=y(b)$ if and only if $x(a) = y(f(a))$. Thererfore, $2^f$ maps $y$ to $y\circ f$, as desired.

Now we move on to define the contravariant functor $2^{(-)}\colon \CHQPol \to \ODQPol$. Fix an enumeration of all finite subsets of $2^{<\IN}$ (the set of finite binary sequences), and for each $a\in\IN$ define $\Phi_a\colon 2^\IN \to 2$ as $\Phi_a(x) = 1$ if and only the $a$-th finite subset of $2^{<\IN}$ contains a prefix of $x$. Thus $(\Phi_a)_{a\in\IN}$ is an enumeration of all continuous functions from $2^\IN$ to $2$ (equivalently, all clopen subsets of $2^\IN$).

\begin{definition}
Define the contravariant functor $\calE\colon \CHQPol \to \ODQPol$ as follows:
\begin{itemize}
\item
For $\equiv$ in $\CHQPol_\Obj$, define $\calE_\Obj(\equiv)$ to be the object $\equiv'$ in $\ODQPol_\Obj$ defined as
\[a\equiv' b \iff (\forall x,y\in 2^\IN)[x\equiv y \Rightarrow \Phi_a(x) = \Phi_b(y)].\]
\item
For $\langle G, \equiv_\src, \equiv_\tar\rangle$ in $\CHQPol_\Mor$, define $\calE_\Mor(\langle G, \equiv_\src, \equiv_\tar\rangle)$ to be the morphism $\langle G', \equiv'_\tar, \equiv'_\src\rangle$ in $\ODQPol_\Mor$ defined as:
\begin{itemize}
\item
$\equiv'_\src = \calE_\Obj(\equiv_\src)$, 
\item
$\equiv'_\tar = \calE_\Obj(\equiv_\tar)$,  and
\item
$G'\in\PL(\IN\times \IN)$ is defined so that $G'(b,a)$ holds if and only if 
\[b \equiv'_\tar b \,\&\, a \equiv'_\src a \,\&\, (\forall x,y\in 2^\IN)[G(x,y) \Rightarrow \Phi_a(x) = \Phi_b(y)]. \]
\end{itemize}
\end{itemize}
\qed
\end{definition}

Again, it is mostly straightforward to show that $\calE$ is a well-defined continuous (computable) contravariant functor. We only show that the 5th condition on $\CHQPol_\Mor$ holds, since it requires slightly more thought than the others. Assuming $b \equiv'_\tar b$, we must find $a\in\IN$ such that $G'(b,a)$. Note that
\begin{eqnarray*}
U &=& \{ x\in 2^\IN \mid \neg(x\equiv_\src x)\}\\
V_0 &=& \{ x\in 2^\IN \mid (\forall y\in 2^\IN)[G(x,y) \Rightarrow \Phi_b(y)=0\}\\
V_1&=&\{ x\in 2^\IN \mid (\forall y\in 2^\IN)[G(x,y) \Rightarrow \Phi_b(y)=1\}
\end{eqnarray*}
are open and $V_0 \cap V_1 \subseteq U$ and $2^\IN\subseteq U \cup V_0\cup V_1$. Since $2^\IN$ is compact  zero-dimensional we can find $a\in\IN$ such that for each $x\in 2^\IN$
\begin{eqnarray*}
\Phi_a(x) = 0 &\Rightarrow& x\in V_0 \cup U,\text{ and }\\
\Phi_a(x) = 1 &\Rightarrow& x\in V_1 \cup U.
\end{eqnarray*}
Note that $a\equiv'_\src a$, because if $x\equiv_\src x'$ then $x,x'\not\in U$ and since $G(x,y)$ and $G(x',y)$ both hold for some $y$, we have that $x$ and $x'$ are both in $V_0$ or  both in $V_1$, hence $\Phi_a(x)=\Phi_a(x')$. Finally, $G'(b,a)$ holds because if $G(x,y)$ then $x\not\in U$ hence $x\in V_0$ or $x\in V_1$, thus $\Phi_a(x)=\Phi_b(y)$.

Proving that the above definition of $\calE\colon \CHQPol \to \ODQPol$ behaves as intended is similar to the case for $\calD\colon \ODQPol\to \CHQPol$, but easier because the resulting spaces are discrete.

By composing we get computable endofunctors $\calE\circ \calD$ on $\ODQPol$ and $\calD\circ \calE$ on $\CHQPol$. We next define computable natural transformations $\eta\colon 1\to \calE\circ \calD$ and $\epsilon\colon 1\to \calD\circ \calE$. In both cases, the intended interpretation of the natural transformations at $X$ is the function $X\to 2^{2^X}$ defined as $x\mapsto \lambda h.h(x)$.  

\begin{definition}\label{def:eta_epsilon}
The natural transformations $\eta\colon 1\to \calE\circ \calD$ and $\epsilon\colon 1\to \calD\circ \calE$ are defined as follows.
\begin{enumerate}
\item
$\eta\colon \ODQPol_\Obj \to \ODQPol_\Mor$ is defined as $\eta(\equiv) = \langle G_\eta, \equiv, \equiv'' \rangle$, where
\begin{itemize}
\item
$\equiv' = \calD_\Obj(\equiv)$,
\item
$\equiv'' = \calE_\Obj(\equiv')$,
\item
$G_\eta(a,b)$ if and only if 
\[a\equiv a \,\&\, b\equiv''b \,\&\, (\forall x\in 2^\IN)[x\equiv ' x \Rightarrow \Phi_b(x) = x(a)].\]
\end{itemize}
\item
$\epsilon\colon \CHQPol_\Obj \to \CHQPol_\Mor$ is defined as $\epsilon(\equiv) = \langle G_\epsilon, \equiv, \equiv'' \rangle$, where
\begin{itemize}
\item
$\equiv'=\calE_\Obj(\equiv)$,
\item
$\equiv'' =\calD_\Obj(\equiv')$,
\item
$G_\epsilon(x,y)$ if and only if 
\[x\equiv x \,\&\, y\equiv''y \,\&\, (\forall a\in \IN)[a\equiv' a \Rightarrow y(a)=\Phi_a(x)].\]
\end{itemize}
\end{enumerate}
\qed
\end{definition}

\begin{lemma}\label{lem:compcontfunctors}
$\calD\colon \ODQPol\to \CHQPol$ and $\calE\colon \CHQPol \to \ODQPol$ are computable contravariant functors, and $\eta\colon 1\to \calE\circ \calD$ and $\epsilon\colon 1\to \calD\circ \calE$ are computable natural transformations.
\end{lemma}
\begin{proof}
We prove that $\eta$ is a natural transformation (the proof for $\epsilon$ is almost identical). Let $\eta(\equiv_\src) =  \langle G^\src_\eta, \equiv_\src, \equiv_\src'' \rangle$ and $\eta(\equiv_\tar) = \langle G^\tar_\eta, \equiv_\tar, \equiv_\tar'' \rangle$. Assume $\langle G, \equiv_\src, \equiv_\tar \rangle$ is in $\ODQPol_\Mor$, and let $ \langle G'', \equiv''_\src, \equiv''_\tar \rangle$ be its image under $\calE\circ\calD$. We must show that $G^\tar_\eta \circ G$ equals $G''\circ G^\src_\eta$. 

We first prove $(G^\tar_\eta \circ G)(a,b)$ implies $(G''\circ G^\src_\eta)(a,b)$ for each $a,b\in\IN$. Assume $c\in \IN$ satisfies $G(a,c)$ and $G^\tar_\eta(c,b)$, and fix $d\in\IN$ satisfying $G^\src_\eta(a,d)$. We must show that $G''(d,b)$ holds, so fix $x,y\in 2^\IN$ and assume $G'(x,y)$. Then $G'(x,y)$ and $G(a,c)$ yields $y(a) = x(c)$, $G^\tar_\eta(c,b)$ and $x\equiv'_\tar x$ yields $\Phi_b(x) = x(c)$, and $G^\src_\eta(a,d)$ and $y\equiv'_\src y$ yields $\Phi_d(y) = y(a)$. Therefore, $\Phi_b(x) = \Phi_d(y)$, hence $G''(d,b)$ holds.

Next we prove $(G''\circ G^\src_\eta)(a,b)$ implies $(G^\tar_\eta \circ G)(a,b)$. Assume $d\in \IN$ satisfies $G^\src_\eta(a,d)$ and $G''(d,b)$, and fix $c\in\IN$ satisfying $G(a,c)$. We must show that $G^\tar_\eta(c,b)$ holds, so fix $x\in 2^\IN$ and assume $x\equiv'_\tar x$. There is $y\in 2^\IN$ satisfying $G'(x,y)$. Then $G'(x,y)$ and $G(a,c)$ yields $y(a) = x(c)$, $G^\src_\eta(a,d)$  and $y\equiv'_\src y$ yields $\Phi_d(y) = y(a)$, and $G''(d,b)$ yields $\Phi_b(x) = \Phi_d(y)$. 
\qed
\end{proof}

\begin{theorem}\label{thrm:adjunction}
$\calD(\eta_X)\circ \epsilon_{\calD(X)} = 1_{\calD(X)}$ and $\calE(\epsilon_Y)\circ \eta_{\calE(Y)} = 1_{\calE(Y)}$ hold for each $X\in\ODQPol_\Obj$ and $Y\in\CHQPol_\Obj$. Therefore, $\calD$ and $\calE$ are adjoint on the right, meaning there is a natural bijection between morphisms of the form $f\colon X\to \calE(Y)$ in $\ODQPol$ and $g\colon Y\to\calD(X)$ in $\CHQPol$ given by $f = \calE_\Mor(g)\circ\eta(X)$ and $g=\calD_\Mor(f)\circ\epsilon(Y)$.
\end{theorem}
\begin{proof}
Fix $Y\in\CHQPol_\Obj$ and PERs $\equiv$, $\equiv'$, $\equiv''$, $\equiv'''$ corresponding to $Y$, $\calE(Y)$, $\calD(\calE(Y))$, $\calE(\calD(\calE(Y)))$, respectively. We prove  $\calE_\Mor(\epsilon(Y))\circ \eta(\calE_\Obj(Y)) = \id(\calE_\Obj(Y))$. $\eta(\calE_\Obj(Y))$ is the morphism  $\langle G_\eta, \equiv', \equiv''' \rangle$, where $G_\eta(a,b)$ if and only if 
\[a\equiv' a \,\&\, b\equiv'''b \,\&\, (\forall y\in 2^\IN)[y\equiv'' y \Rightarrow \Phi_b(y) = y(a)].\]
$\calE_\Mor(\epsilon(Y))$ is the morphism $\langle G'_\epsilon, \equiv''', \equiv'\rangle$ where $G'_\epsilon(b,a)$ if and only if 
\[b \equiv''' b \,\&\, a \equiv' a \,\&\, (\forall x,y\in 2^\IN)[G_\epsilon(x,y) \Rightarrow \Phi_a(x) = \Phi_b(y)], \]
and where $G_\epsilon(x,y)$ if and only if 
\[x\equiv x \,\&\, y\equiv''y \,\&\, (\forall a\in \IN)[a\equiv' a \Rightarrow y(a)=\Phi_a(x)].\]
We must show $a\equiv'a$ implies $(G'_\epsilon\circ G_\eta)(a,a)$. Fix $b\in\IN$ satisfying $G_\eta(a,b)$, and we show $G'_\epsilon(b,a)$. Assume $x,y\in 2^\IN$ satisfy $G_\epsilon(x,y)$. Then $y\equiv'' y$, and from $G_\eta(a,b)$ we obtain $\Phi_b(y) = y(a)$. Furthermore, $G_\epsilon(x,y)$ and $a\equiv' a$ implies $y(a)=\Phi_a(x)$. Therefore, $\Phi_a(x) = \Phi_b(y)$, hence $G'_\epsilon(b,a)$.

The proof that $\calD(\eta_X)\circ \epsilon_{\calD(X)} = 1_{\calD(X)}$ for $X\in\ODQPol_\Obj$ is almost identical and left to the reader.

The final claim is now standard category theory. Fix $f\colon X\to \calE(Y)$ in $\ODQPol$ and define $g=\calD_\Mor(f)\circ\epsilon(Y)$. Then
\begin{eqnarray*}
\calE_\Mor(g)\circ\eta(X) &=&  \calE_\Mor(\calD_\Mor(f)\circ\epsilon(Y))\circ\eta(X)\\
 &=&  \calE_\Mor(\epsilon(Y))\circ\calE_\Mor(\calD_\Mor(f))\circ\eta(X)\text{ (contravariant functor) }\\
 &=&  \calE_\Mor(\epsilon(Y))\circ\eta(\calE_\Obj(Y))\circ f\text{ (natural transformation) }\\
 &=& f.
\end{eqnarray*}
Similarly, given $g\colon Y\to\calD(X)$ in $\CHQPol$, defining $f = \calE_\Mor(g)\circ\eta(X)$ yields $g=\calD_\Mor(f)\circ\epsilon(Y)$.
\qed
\end{proof}

When interpreting $\calD$ and $\calE$ as the contravariant exponential functor $2^{(-)}$, the natural bijection above is the \emph{double exponential transpose} that maps $f\colon X\to 2^Y$ to $g\colon Y\to 2^X$ defined as $g(y) = \lambda x. f(x)(y)$  (see \cite{Taylor02} Proposition~2.11).

The next lemma shows that the dual functors $\calD\colon \ODQPol\to \CHQPol$ and $\calE\colon \CHQPol \to \ODQPol$ restrict to dual functors between $\ODQPol$ and $\ZCHQPol$.

\begin{lemma}\label{lem:DX_is_zerodim}
$\calD_\Obj(\equiv)$ is zero-dimensional for each $\equiv$ in $\ODQPol_\Obj$.
\end{lemma}
\begin{proof}
As before, we write $\equiv'$ for $\calD_\Obj(\equiv)$, and write $\equiv''$ for $ \calE_\Obj(\equiv')$. Assume 
\[(\forall a,b\in \IN)[a \equiv'' b \Rightarrow \Phi_a(x) = \Phi_b(y)],\]
and we show $x\equiv' y$. By defintion of $\equiv'$,
\[x\equiv' y \iff (\forall a,b\in\IN)[a\equiv b \Rightarrow x(a)  = y(b)].\]
So fix $a_0,b_0\in\IN$ with $a_0 \equiv b_0$. We must show $x(a_0)=y(b_0)$. Choose $a,b\in \IN$ with $G_\eta(a_0,a)$ and $G_\eta(b_0,b)$, where $\eta(\equiv)=\langle G_\eta,\equiv,\equiv''\rangle$ as in Definition~\ref{def:eta_epsilon}. Since $a_0\equiv b_0$, we have $a\equiv'' b$, hence $\Phi_a(x) = \Phi_b(y)$ by our assumption on $\equiv'$. Since $x\equiv' x$ and $y\equiv' y$ hold trivially, the definition of $G_\eta$ yields
\[x(a_0) = \Phi_a(x) = \Phi_b(y) = y(b_0),\]
hence $x\equiv' y$.
\qed
\end{proof}


\section{Monads and algebras}\label{sec:monad}


\begin{definition}
Let $\calC$ be an effective quasi-Polish category. A \emph{computable monad} on $\calC$ is a triple $(T,\eta,\mu)$  consisting of:
\begin{itemize}
\item
A computable functor $T\colon \calC \to\calC$,
\item
A computable natural transformation $\eta\colon 1\to T$,
\item
A computable natural transformation $\mu\colon T^2 \to T$,
\end{itemize}
satisfying the following for each $X \in\calC_\Obj$:
\begin{itemize}
\item
$\mu(X)\circ T_\Mor(\mu(X)) = \mu(X)\circ \mu(T_\Obj(X))$,
\item
$\mu(X)\circ T_\Mor(\eta(X)) = \mu(X) \circ \eta(T_\Obj(X)) = \id(T_\Obj(X))$.
\end{itemize}
\qed
\end{definition}

\begin{definition}\label{def:monadalgebra}
Let $\calC$ be an effective quasi-Polish category, and $(T,\eta,\mu)$ a computable monad on $\calC$. Define the \emph{Eilenberg-Moore category} $\Alg_T(\calC)$ as follows:
\begin{itemize}
\item
$\Alg_T(\calC)_\Obj$ is the subspace of $\calC_\Obj \times \calC_\Mor$ of pairs $(A,\alpha)$ satisfying:
\begin{itemize}
\item
$\src_\calC(\alpha) = T_\Obj(A)$ and $\tar_\calC(\alpha)= A$,
\item
$\alpha\circ\eta(A) = \id_\calC(A)$.
\item
$\alpha\circ T_\Mor(\alpha) = \alpha \circ \mu(A)$.
\end{itemize}
\item
$\Alg_T(\calC)_\Mor$ is the subspace of $\calC_\Mor \times \Alg(\calC)_\Obj \times \Alg(\calC)_\Obj$ of all triples $(h,(A,\alpha),(B,\beta))$ satisfying $h\circ \alpha = \beta\circ T_\Mor(h)$.
\item
Composition and identity morphisms are inherited from $\calC$.
\end{itemize}
\qed
\end{definition}

We take a moment to verify this construction can be carried out within $\ACA$. Assuming we have transitive relations encoding $\calC_\Obj$ and $\calC_\Mor$, we can construct a transitive relation encoding their product $\calC_\Obj \times \calC_\Mor$ (see \cite{dbr20} Section~3.1). Now define $\ell,r\colon \calC_\Obj \times \calC_\Mor \to \calC_\Obj\times \calC_\Obj\times \calC_\Obj\times \calC_\Obj$ as 
\begin{eqnarray*}
\ell(A,\alpha) &=& \langle \src_\calC(\alpha), \tar_\calC(\alpha), \alpha\circ T_\Mor(\alpha), \alpha\circ\eta(A)\rangle, \text{ and }\\
r(A,\alpha)&=&\langle T_\Obj(A), A, \alpha \circ \mu(A), \id_\calC(A)\rangle.
\end{eqnarray*}
Then within $\ACA$ we can define a transitive relation encoding $\Alg_T(\calC)_\Obj$ by computing the equalizer of $\ell$ and $r$ (see \cite{dbr20} Section~3.3). $\Alg_T(\calC)_\Mor$ is handled similarly.

\begin{lemma}
If $(T,\eta,\mu)$ is a computable monad on an effective quasi-Polish category $\calC$, then the Eilenberg-Moore category $\Alg_T(\calC)$ is an effective quasi-Polish category.
\qed
\end{lemma}

The composition $\calE_\Mor\circ\epsilon \circ \calD_\Obj$ yields a computable natural transformation $\calE\epsilon\calD\colon\calD\calE\calD\calE \to \calD\calE$, and the composition $\calD_\Mor\circ\eta\circ \calE_\Obj$ yields a computable natural transformation $\calD\eta\calE\colon \calE\calD\calE\calD\to\calE\calD$. The next lemma is the well-known method of constructing monads by composing adjoint functors. The monad axioms can be verified directly using the definitions of (contravariant) functors and natural transformations, along with the identities proven in Theorem~\ref{thrm:adjunction}.

\begin{lemma}\label{lem:monads}
$(\calE\circ \calD, \eta,  \calE\epsilon\calD)$ is a computable monad on $\ODQPol$. $(\calD\circ \calE, \epsilon, \calD\eta\calE)$ is a computable monad on $\CHQPol$, and restricts to a computable monad on $\ZCHQPol$.
\qed
\end{lemma}

To simplify notation, in the following we will omit suffixes and write $\Alg(\ODQPol)$,  $\Alg(\CHQPol)$, and $\Alg(\ZCHQPol)$ for the Eilenberg-Moore categories.

\begin{lemma}\label{lem:algch=algzch}
$\Alg(\CHQPol)=\Alg(\ZCHQPol)$.
\end{lemma}
\begin{proof}
The algebras of the monad $\calD\circ \calE$ are zero-dimensional because they are retracts of zero-dimensional spaces by Lemma~\ref{lem:DX_is_zerodim}.
\qed
\end{proof}

\section{Boolean algebras}\label{sec:boolalgebra}


\subsection{Category of Boolean algebra objects}

A \emph{Boolean algebra} is a tuple $(A,\vee,\wedge,\neg,0,1)$, where $A$ is a set, $\vee$ and $\wedge$ are binary operators on $A$, $\neg$ is a unary operator on $A$, and $0$ and $1$ are elements of $A$, subject to the following axioms (see \cite{DP02}):
\begin{itemize}
\item
(associative) $a \vee (b\vee c) = (a\vee b)\vee c$ and  $a \wedge (b\wedge c) = (a\wedge b)\wedge c$,
\item
(commutative) $a \vee b = b \vee a$ and $a \wedge b = b \wedge a$,
\item
(idempotency) $a \vee a = a$ and $a\wedge a = a$,
\item
(absorption) $a\vee (a\wedge b) = a$ and $a\wedge (a\vee b)=a$,
\item
(distributivity) $a \wedge (b\vee c) = (a\wedge b) \vee (a \wedge c)$ and $a\vee(b\wedge c) = (a\vee b)\wedge (a\vee c)$,
\item
(identity) $a\vee 0 = a$ and $a \wedge 1 = a$,
\item
(complement) $a\vee \neg a = 1$ and $a\wedge \neg a = 0$.
\end{itemize}
Given Boolean algebras $(A,\vee,\wedge,\neg,0,1)$ and $(B,\vee',\wedge',\neg',0',1')$, a morphism $h\colon A\to B$ is a \emph{Boolean algebra homomorphism} if $h$ preserves the Boolean algebra structure (i.e., $h(a\vee b) = h(a)\vee' h(b)$, $h(a\wedge b) = h(a)\wedge' h(b)$, $h(\neg a) = \neg' a$, $h(0) = 0'$, and $h(1) = 1'$).

The two element Boolean algebra $(2,\vee,\wedge,\neg,0,1)$, which has underlying set $2=\{0,1\}$, plays a central role in Stone duality.

\begin{definition}
Given a quasi-Polish category $\calC$ with finite products, we define $\BA(\calC)$ to be the category of Boolean algebra objects in $\calC$.
\begin{itemize}
\item
$\BA(\calC)_\Obj$ is the subspace of $\calC_\Obj \times (\calC_\Mor)^5$ of tuples $(A,\vee,\wedge,\neg,0,1)$, where the morphisms $\vee,\wedge \colon A\times A \to A$, $\neg \colon A \to A$, and $0,1\colon 1_\calC\to A$, satisfy the Boolean algebra axioms (interpreted as commutative diagrams in the usual way).
\item
$\BA(\calC)_\Mor$ is the subspace of $\calC_\Mor \times \BA(\calC)_\Obj \times \BA(\calC)_\Obj$ of all triples $(h,(A,\vee,\wedge,\neg,0,1),(B,\vee',\wedge',\neg',0',1'))$ such that $h\colon A\to B$ is a Boolean algebra homomorphism.
\item
Composition and identity morphisms are inherited from $\calC$.
\end{itemize}
\qed
\end{definition}

Note that the forgetful functor from $\BA(\calC)$ to $\calC$, which projects objects and morphisms onto their first coordinate, is computable. If $\calC$ is an effective quasi-Polish category and finite products are computable, then $\BA(\calC)$ is also an effective quasi-Polish category. In particular, we have the following.

\begin{lemma}
$\BA(\ODQPol)$, $\BA(\CHQPol)$, and $\BA(\ZCHQPol)$ are effective quasi-Polish categories.
\qed
\end{lemma}

See \cite{PS68} or  Corollary VI-4.11 \cite{J82} for the following.

\begin{lemma}\label{lem:bach=bazch}
$\BA(\CHQPol)=\BA(\ZCHQPol)$.
\qed
\end{lemma}

\subsection{Enough Boolean algebra homomorphisms}

The purpose of this section is to prove the following theorem.

\begin{theorem}\label{thm:enuf_homs}
Let $A$ be a Boolean algebra in $\ODQPol$ or $\ZCHQPol$. If $a,b\in A$ and $a\not= b$, then there is a continuous Boolean algebra homomorphism $h\colon A\to 2$ such that $h(a)\not=h(b)$.
\qed
\end{theorem}

The homomorphism $h$ is not necessarily computable from $a$ and $b$, but that will not effect the computability of the duality theorems. We will use Theorem~\ref{thm:enuf_homs} frequently in Section~\ref{sec:algebras_equiv} to prove two elements to be equal. It will also be used in the proof of Theorem~\ref{thm:spatial} when showing that a particular computable morphism is bijective, hence a computable isomorphism by Corollaries~\ref{cor:bij_iso} and \ref{cor:bij_iso_ch}. Since we are not concerned with the computability of $h$ but only its existence, we will avoid encodings and work with the corresponding spaces (Theorem~\ref{thrm:ODQPol} and Theorem~\ref{thrm:CHQPol}).

The core of the proof of Theorem~\ref{thm:enuf_homs} is separated into Lemma~\ref{lem:ods_enuf_homs} (for $\ODQPol$) and Lemma~\ref{lem:zchs_enuf_homs} (for $\ZCHQPol$), which we prove in the following subsections. Those lemmas say that if $c\in A$ does not equal $0$ then there is a continuous Boolean algebra homomorphism $h\colon A\to 2$ with $h(c)=1$. The theorem follows by taking $c=(a\wedge\neg b) \vee (\neg a\wedge b)$, which is not equal to $0$ when $a$ and $b$ are distinct, because any Boolean algebra homomorphism $h\colon A\to 2$ satisfying $h(c)=1$ must have $h(a)\not=h(b)$.

\subsubsection{Boolean algebra homomorphisms in $\ODQPol$}

Note that the homomorphism in the following lemma may not be computable.

\begin{lemma}\label{lem:ods_enuf_homs}
Let $A$ be a Boolean algebra in $\ODQPol$. If $a\in A$ and $a \not= 0$ then there is a Boolean algebra homomorphism $h\colon A \to 2$ such that $h(a)=1$.
\end{lemma}
\begin{proof}
The proof is standard. Let $(a_i)_{i\in\IN}$ be an enumeration of $A$. Set $S_0=\{a\}$. For $i\geq 0$, if $a_i \wedge \bigwedge S_i \not=0$ then set $S_{i+1} = S_i \cup\{a_i\}$, otherwise set $S_{i+1} = S_i \cup\{\neg a_i\}$. It is shown by induction that $\bigwedge S_i \not= 0$ for all $i\in\IN$. Let $S = \bigcup_{i\in\IN} S_i$. Since for each $b\in A$, either $b \in S$ or $\neg b \in S$, and since the meet of any finite subset of $S$ is not equal to $0$, it is easy to see that $S$ is an ultrafilter containing $a$. Defining $h(b)=1 \iff b\in S$ satisfies the lemma. 
\qed
\end{proof}

\subsubsection{Boolean algebra homomorphisms in $\ZCHQPol$}

The methods used in this subsection are derived from \cite{PS68} and Chapter~VI of \cite{g03}. Although we will not go into much detail, we  take care to make sure our proofs can be formalized in $\ACA$ (see \cite{Simpson09}).

Let $A$ be a Boolean algebra in $\ZCHQPol$. For any non-empty clopen set $U\subseteq A$ define $U^+ = \{ y\in A \mid (\exists x\in U)\, x\leq y\}$. Then $U^+$ is clopen, because $U^+ = \pi_1(\wedge^{-1}(U))$, where $\wedge \colon A\times A \to A$ is the meet operation and $\pi_1 \colon A\times A\to A$ is projection onto the first coordinate.

\begin{lemma}\label{lem:zch_alg_seq_complete}
If $A$ is a Boolean algebra in $\ZCHQPol$, then every sequence $(x_i)_{i\in\IN}$ in $A$ has both a join $\bigvee_{i\in\IN} x_i$ and a meet $\bigwedge_{i\in\IN} x_i$. If $U\subseteq A$ is an open upper (lower) set, then $\bigvee_{i\in\IN} x_i \in U$ ($\bigwedge_{i\in\IN} x_i \in U$) if and only if $\bigvee_{i\leq n} x_i \in U$ ($\bigwedge_{i\leq n} x_i \in U$) for some $n\in\IN$.
\end{lemma}
\begin{proof}
Let $(x_i)_{i\in\IN}$ be a sequence in $A$, and define $y_i = x_0 \vee \cdots \vee x_i$. Since $A$ is compact, $(y_i)_{i\in\IN}$ contains a subsequence converging to some $y\in A$. Since the join of any subsequence equals the join of the whole sequence, we will assume without loss of generality that $(y_i)_{i\in\IN}$ converges to $y$. For any $k\in\IN$, we have that $(x_k \vee y_i)_{i\in\IN}$ converges to $x_k \vee y$ by continuity of the join operation. But $x_k \vee y_i = y_i$ for large enough $i$, hence $(x_k \vee y_i)_{i\in\IN}$ converges to $y$. Thus $x_k \vee y = y$ because $A$ is Hausdorff, hence $x_k \leq y$. Therefore, $y$ is an upper bound of $(x_i)_{i\in\IN}$. If $z$ is any other upper bound, then $y_i \in \dnarw z$, for all $i$, and since $\dnarw z$ is closed, the limit point $y$ is in $\dnarw z$. Therefore, $y$ is the join of $(x_i)_{i\in\IN}$. If $U$ is an open upper set, then since $y$ is the limit of (a subsequence of) the ascending sequence $(y_i)_{i\in\IN}$, it is clear that $y\in U$ if and only if $y_i\in U$ for some $i\in\IN$. Meets are handled similarly.
\qed
\end{proof}

\begin{lemma}\label{lem:zch_alg_complete}
If $A$ is a Boolean algebra in $\ZCHQPol$, then every clopen subset $C$ of $A$ has a join $\bigvee C$ and a meet $\bigwedge C$ in $A$. If $U$ is an open upper (lower) set, then $\bigvee C \in U$ ($\bigwedge C \in U$) if and only if there exists finite $F\subseteq C$ such that $\bigvee F \in U$ ($\bigwedge F \in U$).
\end{lemma}
\begin{proof}
Let $U\subseteq A$ be clopen. If $U$ is empty then its join is $\bot$, so there is nothing to prove. Assume $U$ is non-empty. Let $(C_i)_{i\in\IN}$ be an enumeration of all clopen subsets of $A$ that intersect $U$. Fix $x_i\in U\cap C_i$ for each $i\in \IN$.

Let $y$ be the join of $(x_i)_{i\in\IN}$. Clearly, any upper bound of $U$ is above $y$. If $x \in U$ was not less than $y$, then there would be clopen neighborhood $C_i$ of $x$ that misses the closed set $\dnarw y$, but then $x_i\not\leq y$, a contradiction. Therefore, $y$ is the join of $U$. The rest of the lemma is the same as Lemma~\ref{lem:zch_alg_seq_complete}. Meets are handled similarly.
\qed
\end{proof}

\begin{lemma}
Let $A$ be a Boolean algebra in $\ZCHQPol$, let $U\subseteq A$ be an open upper set, and $x\in U$. Then there is $y\in U$ with $x \in int(\uparw y)$.
\end{lemma}
\begin{proof}
Let $V\subseteq U$ be a clopen neighborhood of $x$. Let $Q$ be the preimage of $V$ under the map $y \mapsto x\wedge y$. Note that $Q\subseteq U$ because $U$ is an upper set, and that $Q$ is clopen and contains $x$. Set $W = \{ y\in Q \mid (\forall z\in Q)\, y\wedge z \in Q\}$. Then $W$ is clopen:
\begin{itemize}
\item
$W$ is open because it is defined by universal quantification over the compact set $Q$ by an open predicate.
\item
$W$ is closed because $y\not\in W$ implies there is $z\in Q$ such that $y$ is in the preimage of $\neg  Q$ under the map $y\mapsto y\wedge z$. The preimage of $\neg Q$ under that map is a clopen set that is disjoint from $W$.
\end{itemize}
Also note that $x \in W$. Let $y = \bigwedge W$. Then $y\in Q$, because $Q$ is closed and the meet of any finite subset of $W$ is in $Q$. Then $y\in U$ and $x \in W \subseteq \uparw y$.
\qed
\end{proof}

The above lemma implies there is clopen $C\subseteq U$ such that $x \in C \subseteq \uparw y\subseteq U$. Therefore, if $U$ is an open upper set containing $x$, then there exists a clopen nieghborhood $C$ of $x$ satisfying $\bigwedge C \in U$.

\begin{lemma}\label{lem:zchs_enuf_homs}
If $x\in A$ and $x \not= 0$ then there is a continuous Boolean homomorphism $h\colon A \to 2$ such that $h(x)=1$.
\end{lemma}
\begin{proof}
Fix an enumeration $(C_i)_{i\in\IN}$ of all clopen subsets of $A$. Set $x_0 = x$, and for $i\geq 0$, assume $x_i$ has been defined and $x_i \not=0$. If $x_i \wedge \bigwedge C_i \not=0$, then there is $x_{i+1}\not=0$ such  $x_i \wedge \bigwedge C_i \in int(\uparw x_{i+1})$. Otherwise, let $x_{i+1}\not=0$ be such that $x_i \in int(\uparw x_{i+1})$. Then $W = \bigcup_{i\in\IN} \uparw x_i$ is an open filter containing $x$ that does not contain $0$.

For any $a\in A$, $W' = \{ z \mid \neg a \vee z \in W\}$ is an open upper set containing $a$, so there exists $i\in\IN$ such that $a \in C_i$ and $\bigwedge C_i \in W'$. If $x_i \wedge \bigwedge C_i \not= 0$ then $a \in C_i \subseteq \uparw x_{i+1} \subseteq W$ hence $a\in W$. Otherwise, since $x_i \in W$ and $\neg a \vee \bigwedge C_i \in W$ and
\begin{eqnarray*}
x_i \wedge (\neg a \vee \bigwedge C_i) &= & (x_i \wedge \neg a) \vee (x_i \wedge \bigwedge C_i)\\
&=& (x_i \wedge \neg a)\vee 0\\
&\leq& \neg a,
\end{eqnarray*}
we obtain $\neg a \in W$. Therefore, for each $a\in A$, exactly one of $a$ and $\neg a$ is in $W$, hence $U = \{ a\in A \mid \neg a \in W\}$ is the complement of $W$. The continuity of $\neg$ implies $U$ is open, hence $W$ is clopen. Furthermore, if $a\vee b \in W$, then $\neg a \wedge \neg b = \neg(a\vee w) \not\in W$, hence $\neg a \not\in W$ or $\neg b \not\in W$, which implies $a\in W$ or $b\in W$. Therefore, $W$ is a clopen prime filter containing $x$, hence $h\colon A\to 2$ defined as $h(a)=1 \iff a\in W$ satisfies the lemma.
\qed
\end{proof}


\section{Restricted $\lambda$-calculus}\label{sec:lambdacalculus}


In this section we introduce a $\lambda$-calculus that will make it easier to define and reason with morphisms in $\ODQPol$ and $\CHQPol$. However, we can not use the full $\lambda$-calculus because the categories we work with  are not cartesian closed. We only develop the $\lambda$-calculus enough so that we can prove the Stone duality theorems, and we leave a more complete development for future research.

The restricted $\lambda$-calculus we describe in this section, and its applications in later sections, is inspired by P.~Taylor's restricted $\lambda$-calculus for Abstract Stone Duality (ASD) (see \cite{Taylor02}). The main difference is that we will work with two separate categories, whereas Taylor develops ASD within a single category of locally compact (or more general) ``spaces''. We keep the categories separate because $\ODQPol$ and $\CHQPol$ have good category theoretical properties (they are both pretoposes, the first with inductive types and the second with coinductive types), and these properties are lost when we embed them into the common category of locally compact spaces. Taylor's work is also more concerned with the Sierpinski space $\sierp$ as the dualizing object, whereas we focus on the overt discrete compact Hausdorff space $2$ with two elements.

\subsection{Type Formation rules}

In this paper, we will differentiate between the words \emph{sort} and \emph{type}. A \emph{sort} is a collection of types, and for our purposes will refer to the categories $\ODQPol$ and $\CHQPol$. To avoid writing each rule twice, we use $\Sort$ as a metavariable standing for either $\ODQPol$ or $\CHQPol$, and $\dualSort$ to refer to the other sort. 

Each sort contains two kinds of types. An \emph{object type} is an object of a given sort, and a \emph{morphism type} is a morphism of a given sort. Each \emph{term} will be assigned a type. We use the notation $e:X$ to  mean that $e$ is of object type $X$, and the notation $e:[X\to Y]$ to mean that $e$ is of morphism type $X\to Y$. We use double colons such as $X::\Sort$ and $[X\to Y]::\Sort$ to show the sort of a type. To save space we will often combine the two declarations, such as $e:X::\Sort$.

Below are the object type formation rules:
\begin{enumerate}
\item
\AxiomC{($X\in \ODQPol_\Obj$)}
\UnaryInfC{$X::\ODQPol$}
\DisplayProof
\qquad
\AxiomC{($X\in \CHQPol_\Obj$)}
\UnaryInfC{$X::\CHQPol$}
\DisplayProof
\item
\AxiomC{}
\UnaryInfC{$1::\Sort$}
\DisplayProof
\item
\AxiomC{$X::\Sort$}
\AxiomC{$Y::\Sort$}
\BinaryInfC{$X\times Y::\Sort$}
\DisplayProof
\item
\AxiomC{$X::\Sort$}
\UnaryInfC{$2^X::\dualSort$}
\DisplayProof

As a special case, by taking $X$ to be the terminal object $1$, we obtain $2^1::\Sort$, which we abbreviate to $2::\Sort$.
\end{enumerate}

Morphism types have only one formation rule (note that $X$ and $Y$ in the following are object types, and not morphism types):
\begin{enumerate}
\item
\AxiomC{$X::\Sort$}
\AxiomC{$Y::\Sort$}
\BinaryInfC{$[X\to Y]::\Sort$}
\DisplayProof
\end{enumerate}

A more complete treatment should include other object types compatible with these categories, such as coproducts, inductive types for $\ODQPol$, coinductive types for $\CHQPol$, as well as certain subtypes and quotient types, but we leave that for future research.

\subsection{Term formation rules}

A finite set of variable-type-sort triples of the form $x:X::\Sort$ or $x:X::\dualSort$ is called a \emph{context}, and is usually denoted by the letter $\Gamma$. A context can only contain object types, and can not contain morphism types.

The term formation rules below apply to both $\Sort$ and $\dualSort$, and can be mixed within a single derivation. The double line in the (dual) rule means it can be applied in both direction (up or down). 

It is important to note that an object type variable (e.g. $X$) cannot be pattern matched with a function type variable (e.g. $[X\to Y]$), and vice versa. The (dual-mor) rule acts as a restricted means of converting between object and morphism types via the dual category, thereby increasing the expressiveness of the calculus (see the sample derivations below).

For clarity, we sometimes omit the sorts of types when they are known from context.

\indent

\noindent
\makebox[\linewidth][c]{
\fbox{
\renewcommand{\arraystretch}{4}
\begin{tabular}{ll}
(const-obj)&
\AxiomC{($c \in X$ and $X$ in $\Sort_\Obj$)}
\UnaryInfC{$\Gamma\vdash c:X::\Sort$}
\DisplayProof
\\
(const-mor)&
\AxiomC{($f:X\to Y$ in $\Sort_\Mor$)}
\UnaryInfC{$\Gamma\vdash f:[X\to Y]::\Sort$}
\DisplayProof
\\
(var)&
\AxiomC{$(X :: \Sort)$}
\UnaryInfC{$\Gamma, x:X \vdash x:X::\Sort$}
\DisplayProof
\\
(dual-obj)&
\AxiomC{$\Gamma\vdash e:2::\Sort$}
\doubleLine
\UnaryInfC{$\Gamma\vdash e:2::\dualSort$}
\DisplayProof
\\
(dual-mor)&
\AxiomC{$\Gamma\vdash e:[X\to 2]::\Sort$}
\doubleLine
\UnaryInfC{$\Gamma\vdash e:2^X::\dualSort$}
\DisplayProof
\end{tabular}
\begin{tabular}{ll}
(app)&
\AxiomC{$\Gamma\vdash e_1:[X\to Y]::\Sort$}
\AxiomC{$\Gamma \vdash e_2:X::\Sort$}
\BinaryInfC{$\Gamma \vdash e_1(e_2):Y::\Sort$}
\DisplayProof
\\
(curry)&
\AxiomC{$\Gamma,x:X::\Sort \vdash e : Y::\Sort$}
\UnaryInfC{$\Gamma \vdash \lambda x:X.e: [X\to Y]::\Sort$}
\DisplayProof
\\
(pair)&
\AxiomC{$\Gamma \vdash e_1 : X::\Sort$}
\AxiomC{$\Gamma \vdash e_2 : Y::\Sort$}
\BinaryInfC{$\Gamma \vdash \langle e_1,e_2\rangle :X\times Y::\Sort$}
\DisplayProof
\\
(proj)&
\AxiomC{$\Gamma \vdash e :X\times Y::\Sort$}
\UnaryInfC{$\Gamma \vdash \pi_1(e) : X::\Sort$}
\DisplayProof
\quad
\AxiomC{$\Gamma \vdash e :X\times Y::\Sort$}
\UnaryInfC{$\Gamma \vdash \pi_2(e) : Y::\Sort$}
\DisplayProof
\end{tabular}
}
}

\subsection{Equational rules}

We assume the standard $\alpha$, $\beta$, and $\eta$ rules for reasoning about the equality of terms. 
\begin{itemize}
\item[($\alpha$)]
Terms are equal if they only differ in names of bounded variables.
\item[($\beta$)]
\begin{enumerate}
\item[i.]
$(\lambda x:X.e_1)(e_2) = e_1[e_2/x]$, where $e_1[e_2/x]$ is the result of substituting all free occurrences of $x$ in $e_1$ with the term $e_2$ (while avoiding variable capture).
\item[ii.]
$\pi_1(\langle e_1,e_2\rangle) = e_1$
\item[iii.]
$\pi_2(\langle e_1,e_2\rangle)=e_2$
\end{enumerate}
\item[($\eta$)]
\begin{enumerate}
\item[i.]
$f = \lambda x:X. f(x)$
\item[ii.]
$e = \langle \pi_1(e), \pi_2(e) \rangle$ (for $e: X\times Y$)
\end{enumerate}
\end{itemize}

\subsection{Interpretation}


\subsubsection{Interpretation of object types}

Each object type $X$ will be interpreted as an object $\llbracket X \rrbracket$ of the category corresponding to its sort. The interpretation of object types is as follows:

\begin{enumerate}
\item
\AxiomC{($X\in \ODQPol_\Obj$)}
\UnaryInfC{$\llbracket X \rrbracket = X$}
\DisplayProof
\qquad
\AxiomC{($X\in \CHQPol_\Obj$)}
\UnaryInfC{$\llbracket X \rrbracket  = X$}
\DisplayProof
\item
$\llbracket 1 \rrbracket$ is the terminal object, which we also write as $1$.
\item
\AxiomC{$X::\Sort$}
\AxiomC{$Y::\Sort$}
\BinaryInfC{$\llbracket X\times Y \rrbracket = \llbracket X \rrbracket\times \llbracket Y \rrbracket$}
\DisplayProof
\item
\AxiomC{($X::\ODQPol$)}
\UnaryInfC{$\llbracket 2^X \rrbracket=\calD_\Obj(\llbracket X \rrbracket) \in \CHQPol_\Obj$}
\DisplayProof
\qquad
\AxiomC{($X::\CHQPol$)}
\UnaryInfC{$\llbracket 2^X \rrbracket =\calE_\Obj(\llbracket X \rrbracket) \in \ODQPol_\Obj$}
\DisplayProof
\end{enumerate}

\subsubsection{Interpretation of terms}

A context of the form
\[\Gamma=\{x_1:X_1::\ODQPol,\ldots,x_m:X_m::\ODQPol,y_1:Y_1::\CHQPol,\ldots,y_n:Y_n::\CHQPol\}\]
is mapped to a space of the form $\llbracket \Gamma \rrbracket = \IN^m \times (2^\IN)^n$, with the intention that elements of $\llbracket \Gamma \rrbracket$ represent choices of values for the corresponding variables. If $\Gamma$ is empty, then $\llbracket \Gamma \rrbracket$ is the singleton space. Given an object type $\llbracket X \rrbracket$, we write $\equiv_X$ for the corresponding PER. Recall that in both categories, the terminal object is the PER where all elements are equivalent.

In typical semantics, a term in context $\Gamma \vdash e:X$ is interpreted as a morphism $\llbracket \Gamma \rrbracket \to \llbracket X \rrbracket$ in the category. However, we must take a different approach, because in our case the context contains objects from different categories, and such a morphism might not exist in either category. 

Instead, a term in context of the form $\Gamma \vdash e$ will be interpreted as a (partial) function $\llbracket \Gamma\vdash e\rrbracket \colon \llbracket \Gamma \rrbracket\to \Sort_\Mor$, where $\Sort$ is the sort of the type of $e$. The output of the function $\llbracket \Gamma\vdash e\rrbracket(a_1,\ldots,a_m,b_1,\ldots,b_n)$ is only meaningful when the inputs satisfy $a_i \equiv_ {X_i} a_i$ (for $1\leq i\leq m$ ) and $b_j \equiv_{Y_j} b_j$ (for $1 \leq j \leq n$). Furthermore, the interpretation is invariant in the sense that $\llbracket \Gamma\vdash e\rrbracket(a'_1,\ldots,a'_m,b'_1,\ldots,b'_n) = \llbracket \Gamma\vdash e\rrbracket(a_1,\ldots,a_m,b_1,\ldots,b_n)$ whenever $a'_i \equiv_ {X_i} a_i$ and $b'_j \equiv_{Y_j} b_j$. 

When $\llbracket \Gamma\vdash e \rrbracket$ is given valid inputs, a term $\Gamma \vdash e:X::\Sort$ with an object type will be mapped to a morphism from the terminal object $1$ to $\llbracket X \rrbracket$, and a term $\Gamma \vdash e:[X\to Y]::\Sort$ with a morphism type will be mapped to a morphism from $\llbracket X \rrbracket$ to $\llbracket Y \rrbracket$. 

Here and in the following, $S=\IN$ and $\calH=\calE_\Obj$ and $\calH^*=\calD_\Obj$ when $\Sort=\ODQPol$, and $S = 2^\IN$ and $\calH=\calD_\Obj$ and $\calH^* = \calE_\Obj$ when $\Sort=\CHQPol$. Thus, the interpretation of $2$ in $\Sort$ is $\llbracket 2 \rrbracket = \llbracket 2^1 \rrbracket = \calH(\llbracket 1 \rrbracket)$.

The interpretation is defined recursively according to the term formation rules as follows (below we will assume $\Gamma$ contains $n$ variables):

\begin{itemize}
\item
(const-obj): Given $c\in S$ with $c \equiv_X c$, $\llbracket \Gamma\vdash c:X \rrbracket$ is the constant function to the morphism $\langle G,\equiv_1, \equiv_X\rangle$ with graph $G(a,b) \iff b \equiv_X c$.
\item
(const-mor): Given $f\colon X\to Y$ in $\Sort_\Mor$, $\llbracket \Gamma\vdash f \rrbracket$ is the constant function with value $f$.
\item
(var): $\llbracket \Gamma,x:X \vdash x:X \rrbracket \colon \llbracket \Gamma \rrbracket\times S \to \Sort_\Mor$ is defined so that $\llbracket \Gamma,x:X \vdash x:X \rrbracket(z_1,\ldots,z_n,c)$ is the morphism $\langle G,\equiv_1, \equiv_X\rangle$ with graph $G(a,b) \iff b \equiv_X c$.
\item
(app): $\llbracket \Gamma\vdash e_1(e_2) \rrbracket (z_1,\ldots,z_n) = \llbracket \Gamma\vdash e_1\rrbracket(z_1,\ldots,z_n) \circ \llbracket \Gamma\vdash e_2 \rrbracket(z_1,\ldots,z_n)$, where $\circ$ is composition in $\Sort$.
\item
(curry): $\llbracket \Gamma\vdash \lambda x:X.e:[X\to Y] \rrbracket (z_1,\ldots,z_n)$ is the morphism $\langle G, \equiv_X,\equiv_Y\rangle$ with graph $G(a,b)$ if and only if $a\equiv_X a$ and $\langle a,b\rangle$ is in the graph of the morphism $\llbracket \Gamma,x:X \vdash e\rrbracket (z_1,\ldots,z_n,a)$. To see that this is well-defined, note that for valid inputs $\llbracket \Gamma,x:X \vdash e:Y\rrbracket (z_1,\ldots,z_n,a)$ is assumed to be a morphism $\langle H,\equiv_1,\equiv_Y\rangle$. In this case, $a$ is a valid input if and only if $a \equiv_X a$. The PER for the terminal object satisfies $a \equiv_1 a$, hence there is $b$ with $H(a,b)$. This shows that there is $b$ such that $G(a,b)$. The other requirements for $\langle G, \equiv_X,\equiv_Y\rangle$ to be a morphism are easily verified.
\item
(dual): For (dual-mor), by assumption $\llbracket \Gamma \vdash e:[X\to 2]\rrbracket(z_1,\ldots,z_n)$ is a morphism $f\colon X\to \calH(1)$ in $\Sort$. Then $\llbracket \Gamma \vdash e:2^X\rrbracket(z_1,\ldots,z_n)$ is defined by applying the bijection in Theorem~\ref{thrm:adjunction} to $f$ to obtain the corresponding morphism $g\colon 1 \to \calH^*(X)$ in $\dualSort$. 

Applying the rule in the other direction is similar. The rule (dual-obj) can be treated as a special case of (dual-mor) with $X=1$. 
\item
(pair): By assumption, $\llbracket \Gamma \vdash e_1 : X \rrbracket$ and $\llbracket\Gamma \vdash e_2 : Y\rrbracket$ on valid inputs return morphisms $f\colon 1\to \llbracket X \rrbracket$ and $g\colon 1\to \llbracket Y\rrbracket$. Then $\llbracket \Gamma \vdash \langle e_1,e_2\rangle :X\times Y \rrbracket$ (on the same inputs) is defined to be the uniquely determined map $\langle f,g\rangle\colon 1\to \llbracket X\rrbracket \times \llbracket Y\rrbracket = \llbracket X\times Y\rrbracket$. 
\item
(proj): This is a special case of invoking (const-mor) for the projection $\pi_1\colon \llbracket X\rrbracket\times\llbracket Y \rrbracket \to \llbracket X\rrbracket$ (or $\pi_2\colon \llbracket X\rrbracket\times\llbracket Y \rrbracket \to \llbracket Y\rrbracket$), followed by an application of (app) to the term $e:X\times Y$.
\end{itemize}

Note that the interpretation of a lambda term is computable if it is formed using only computable instances of (const-obj) and (const-mor).

\subsubsection{Soundness of equational rules}

One complication with the restricted $\lambda$-calculus is that because of the (dual-obj) and (dual-mor) rules, the derivation of a lambda term is not necessarily unique. Here we check that this does not effect the interpretation of the terms, and that the interpretation is sound with respect to $\alpha$, $\beta$, and $\eta$-equivalence. 

To avoid this problem, we define a second interpretation of the terms as morphisms within the category $\QPol$ of quasi-Polish spaces and continuous functions. Let $\calF\colon \ODQPol \to \QPol$ and $\calG\colon \CHQPol \to \QPol$ be the functors from Theorems~\ref{thrm:ODQPol} and \ref{thrm:CHQPol}, but with their codomains extended to the common category $\QPol$. Interpret object types and morphism types as objects of $\QPol$ as follows:
\begin{itemize}
\item
$\llbracket X \rrbracket' = \calF(\llbracket X \rrbracket)$ for each object type $X$ of sort $\ODQPol$,
\item
$\llbracket X \rrbracket' = \calG(\llbracket X \rrbracket)$ for each object type $X$ of sort $\CHQPol$,
\item
$\llbracket [X\to Y] \rrbracket' = \llbracket Y \rrbracket'^{\llbracket X \rrbracket'}$ for each morphism type $[X\to Y]$, where $\llbracket Y \rrbracket'^{\llbracket X \rrbracket'}$ is the space of continuous functions with the compact-open topology.
\end{itemize}
$\QPol$ is not cartesian closed, but its exponentiable objects are precisely the locally compact quasi-Polish spaces, and the topology on the exponential object is the compact-open topology (Theorem~16.3 of \cite{DGJL}). Therefore, the above interpretations are valid within our restricted $\lambda$-calculus. Furthermore, $\llbracket 2^X\rrbracket'$ and $\llbracket [X\to 2] \rrbracket'$ are isomorphic (see the discussion after Definition~\ref{def:functor_D}). 

Under this interpretation, the distinction between sorts, the distinction between object and morphism types, and the (dual) rules can all be ignored, and any derivation in our restricted $\lambda$-calculus becomes a derivation in the simply typed lambda calculus. Therefore, any term in context $\Gamma \vdash e:X$ derived in the restricted $\lambda$-calculus, can be interpreted as a morphism $\llbracket \Gamma\vdash e:X \rrbracket' \colon \llbracket \Gamma \rrbracket' \to \llbracket X \rrbracket'$ in $\QPol$ in the usual way. Clearly, $\alpha$, $\beta$, and $\eta$-equivalence are sound for this interpretation.

Now assume $\Gamma$ is a context containing object types $X_1,\ldots,X_n$. For valid $z_1,\dots, z_n$ (i.e., $z_i \equiv_{X_i} z_i$ for $1\leq i\leq n$) we write $\vec{z} \colon 1 \to \llbracket \Gamma \rrbracket'$ for the morphism in $\QPol$ whose $i$-th projection is $z_i$ for each $i$. Then it can be shown by induction on deriviations in the restricted $\lambda$-calculus that for all valid $\vec{z}$,
\begin{itemize}
\item
$\llbracket \Gamma\vdash e:X \rrbracket'\circ\vec{z} = \calF(\llbracket \Gamma\vdash e:X \rrbracket (z_1,\ldots,z_n))$ when $e$ has an object type of sort $\ODQPol$,
\item
$\llbracket \Gamma\vdash e:[X\to Y] \rrbracket'\circ\vec{z}$ is the curried form of $\calF(\llbracket \Gamma\vdash e:[X\to Y] \rrbracket (z_1,\ldots,z_n))$ when $e$ has a morphism type of sort $\ODQPol$, 
\end{itemize}
and similarly when $e$ has a type of sort $\CHQPol$ (but using $\calG$ instead of $\calF$). Since $\calF$ and $\calG$ are full and faithful embeddings, and $\QPol$ is well-pointed, this shows that the interpretation of a well-formed lambda term does does not depend on its derivation, and that the equational rules are sound.

\subsection{Sample derivations and interpretations}

To reduce clutter, in the following examples we use the following variation of (app) which will allow us to reduce the size of the contexts.

\vspace{0.5cm}

\makebox[\linewidth][c]{
\AxiomC{$\Gamma\vdash e_1:[X\to Y]::\Sort$}
\AxiomC{$\Delta \vdash e_2:X::\Sort$}
\RightLabel{(app')}
\BinaryInfC{$\Gamma, \Delta \vdash e_1(e_2):Y::\Sort$}
\DisplayProof
}

\vspace{0.5cm}

The context $\Gamma,\Delta$ is the union of $\Gamma$ and $\Delta$, not a multiset. Any derivation using (app') can be modified into a deriviation using the original (app) rule by weakening the contexts.

\subsubsection{Sample derivation (composing $f\colon X\to Y$ and $g\colon Y\to Z$)}

\vspace{0.5cm}

\makebox[\linewidth][c]{
\AxiomC{$\vdash g:[Y\to Z]::\Sort$}
\AxiomC{$\vdash f:[X\to Y]::\Sort$}
\AxiomC{}
\RightLabel{(var)}
\UnaryInfC{$x:X \vdash x:X::\Sort$}
\RightLabel{(app')}
\BinaryInfC{$x:X \vdash f(x):Y::\Sort$}
\RightLabel{(app')}
\BinaryInfC{$x:X \vdash g(f(x)):Z::\Sort$}
\RightLabel{(curry)}
\UnaryInfC{$\vdash \lambda x:X.g(f(x)):[X\to Z]::\Sort$}
\DisplayProof
}

\vspace{0.5cm}

The function $\llbracket \vdash \lambda x:X.g(f(x)):[X\to Z]::\Sort \rrbracket $ maps the unique element of the singleton space to the morphism $g\circ f$.

\subsubsection{Sample derivation $(2^f\colon 2^Y\to 2^X)$}

\vspace{0.5cm}

\makebox[\linewidth][c]{
\AxiomC{}
\RightLabel{(var)}
\UnaryInfC{$\varphi:2^Y::\dualSort \vdash \varphi:2^Y::\dualSort$}
\RightLabel{(dual-mor)}
\UnaryInfC{$\varphi:2^Y::\dualSort \vdash \varphi:[Y\to 2]::\Sort$}
\AxiomC{$\vdash f: [X\to Y]::\Sort$}
\AxiomC{}
\RightLabel{(var)}
\UnaryInfC{$x:X::\Sort \vdash x:X::\Sort$}
\RightLabel{(app')}
\BinaryInfC{$x:X::\Sort \vdash f(x):Y::\Sort$}
\RightLabel{(app')}
\BinaryInfC{$\varphi:2^Y::\dualSort,x:X::\Sort \vdash \varphi(f(x)):2::\Sort$}
\RightLabel{(curry)}
\UnaryInfC{$\varphi:2^Y::\dualSort \vdash \lambda x:X.\varphi(f(x)):[X\to 2]::\Sort$}
\RightLabel{(dual-mor)}
\UnaryInfC{$\varphi:2^Y::\dualSort \vdash \lambda x:X.\varphi(f(x)): 2^X::\dualSort$}
\RightLabel{(curry)}
\UnaryInfC{$\vdash \lambda \varphi:2^Y. \lambda x:X.\varphi(f(x)): [2^Y \to 2^X]::\dualSort$}
\DisplayProof
}

\vspace{0.5cm}

We step through the process of interpreting the resulting lambda term using the above derivation (and skipping the easier parts). For concreteness, we will assume $f\colon X\to Y$ is a morphism in $\ODQPol$. We omit the subscripts on the functors to simplify notation.

\begin{itemize}
\item
The function $\llbracket  \varphi:2^Y::\CHQPol \vdash \varphi:2^Y::\CHQPol \rrbracket$ maps  valid $c\in 2^\IN$ to the morphism $\varphi_c\colon \llbracket 1 \rrbracket \to \calD(\llbracket Y\rrbracket)$ with graph $G_c(a,b) \iff b \equiv_{\calD(\llbracket Y\rrbracket)} c$.
\item
$\llbracket  \varphi:2^Y::\CHQPol \vdash \varphi:[Y\to 2]::\ODQPol \rrbracket$ maps valid $c\in 2^\IN$ to the morphism $\calE(\varphi_c)\circ\eta(\llbracket Y\rrbracket)\colon\llbracket Y \rrbracket \to \calE(\llbracket 1 \rrbracket)$.
\item
$\llbracket \varphi:2^Y::\CHQPol,x:X::\ODQPol \vdash \varphi(f(x)):2::\ODQPol\rrbracket$ maps valid $c\in 2^\IN$ and $d\in\IN$ to the morphism $\calE(\varphi_c)\circ\eta(\llbracket Y\rrbracket) \circ f \circ x_d \colon \llbracket 1 \rrbracket \to \calE(\llbracket 1 \rrbracket)$, where $x_d\colon \llbracket 1 \rrbracket \to \llbracket X \rrbracket$ has graph $G_d(a,b) \iff b\equiv_X d$.
\item
$\llbracket\varphi:2^Y::\CHQPol \vdash \lambda x:X.\varphi(f(x)):[X\to 2]::\ODQPol\rrbracket$ maps valid $c\in 2^\IN$ to the morphism $\calE(\varphi_c)\circ\eta(\llbracket Y\rrbracket) \circ f \colon X \to \calE(\llbracket 1 \rrbracket)$. 
\item
$\llbracket\varphi:2^Y::\CHQPol \vdash \lambda x:X.\varphi(f(x)): 2^X::\CHQPol\rrbracket$ maps valid $c\in2^\IN$ to the morphism $\calD\big(\calE(\varphi_c)\circ\eta(\llbracket Y\rrbracket) \circ f\big)\circ \epsilon(\llbracket 1 \rrbracket) \colon \llbracket 1 \rrbracket \to \calD(X)$. Since $\calD$ is contravariant, this is equal to the morphism $\calD(f) \circ \calD(\eta(\llbracket Y\rrbracket)) \circ \calD\calE(\varphi_c)\circ \epsilon(\llbracket 1 \rrbracket)$. By naturality of $\epsilon$, this is equal to $\calD(f) \circ \calD(\eta(\llbracket Y\rrbracket)) \circ \epsilon(\calD(\llbracket Y \rrbracket)) \circ \varphi_c$, and by Theorem~\ref{thrm:adjunction}, this is equal to $\calD(f) \circ \varphi_c$.
\item
Finally, $\llbracket\vdash \lambda \varphi:2^Y. \lambda x:X.\varphi(f(x)): [2^Y \to 2^X]::\CHQPol\rrbracket$ maps the unique element of the singleton space to the morphism $\calD(f)$.
\end{itemize}
Similarly, when $\Sort = \CHQPol$ the term is interpreted as $\calE(f)$.

\subsubsection{Sample derivation $(\eta_X\colon X\to 2^{2^X})$}

\vspace{0.5cm}

\AxiomC{}
\RightLabel{(var)}
\UnaryInfC{$\varphi:2^X::\dualSort \vdash \varphi:2^X::\dualSort$}
\RightLabel{(dual-mor)}
\UnaryInfC{$\varphi:2^X::\dualSort \vdash \varphi:[X\to 2]::\Sort$}
\AxiomC{}
\RightLabel{(var)}
\UnaryInfC{$x:X::\Sort\vdash x:X::\Sort$}
\RightLabel{(app')}
\BinaryInfC{$x:X::\Sort,\varphi:2^X::\dualSort \vdash \varphi(x):2::\Sort$}
\RightLabel{(dual-obj)}
\UnaryInfC{$x:X::\Sort,\varphi:2^X::\dualSort \vdash \varphi(x):2::\dualSort$}
\RightLabel{(curry)}
\UnaryInfC{$x:X::\Sort \vdash \lambda \varphi:2^X.\varphi(x):[{2^X}\to 2]::\dualSort$}
\RightLabel{(dual-mor)}
\UnaryInfC{$x:X::\Sort \vdash \lambda \varphi:2^X.\varphi(x):2^{2^X}::\Sort$}
\RightLabel{(curry)}
\UnaryInfC{$\vdash \lambda x:X.\lambda \varphi:2^X.\varphi(x):[X\to 2^{2^X}]::\Sort$}
\DisplayProof

\vspace{0.5cm}

The lambda term is interpreted as follows. Again, for concreteness we will assume $\Sort = \ODQPol$. We skip the first few steps, which are the same as deriving $2^f$ when $f$ is the identity on $X$.
\begin{itemize}
\item
$\llbracket \varphi:2^X::\CHQPol,x:X::\ODQPol \vdash \varphi(x):2::\ODQPol\rrbracket$ maps valid $c\in 2^\IN$ and $d\in\IN$ to the morphism $\calE(\varphi_c)\circ\eta(\llbracket X\rrbracket) \circ x_d \colon \llbracket 1 \rrbracket \to \calE(\llbracket 1 \rrbracket)$, where $\varphi_c\colon \llbracket 1 \rrbracket \to \calD(\llbracket X\rrbracket)$ has graph $G_c(a,b) \iff b \equiv_{\calD(\llbracket X\rrbracket)} c$, and $x_d\colon \llbracket 1 \rrbracket \to \llbracket X \rrbracket$ has graph $G_d(a,b) \iff b\equiv_X d$.
\item
$\llbracket x:X::\ODQPol,\varphi:2^X::\CHQPol \vdash \varphi(x):2::\CHQPol \rrbracket$ maps valid $c\in 2^\IN$ and $d\in\IN$ to the morphism $\calD\big(\calE(\varphi_c)\circ\eta(\llbracket X\rrbracket) \circ x_d\big)\circ \epsilon(\llbracket 1 \rrbracket) \colon \llbracket 1 \rrbracket \to \calE_\Obj(\llbracket 1 \rrbracket)$. Since $\calD$ is contravariant, this equals $\calD(x_d)\circ\calD(\eta(\llbracket X\rrbracket))\circ \calD\calE(\varphi_c)\circ \epsilon(\llbracket 1 \rrbracket)$. By naturality of $\epsilon$, this is equal to $\calD(x_d)\circ\calD(\eta(\llbracket X\rrbracket))\circ\epsilon(\calD(\llbracket X \rrbracket)) \circ \varphi_c$, and by Theorem~\ref{thrm:adjunction}, this is equal to $\calD(x_d)\circ\varphi_c$.
\item
$\llbracket x:X::\ODQPol \vdash \lambda \varphi:2^X.\varphi(x):[{2^X}\to 2]::\CHQPol\rrbracket$ maps valid $d\in\IN$ to the morphism $\calD(x_d)\colon \calD(\llbracket X \rrbracket)\to\calD(\llbracket 1 \rrbracket)$.
\item
$\llbracket x:X::\ODQPol \vdash \lambda \varphi:2^X.\varphi(x):2^{2^X}::\ODQPol\rrbracket$ maps valid $d\in\IN$ to the morphism $\calE\calD(x_d)\circ\eta(\llbracket 1 \rrbracket) \colon \llbracket 1 \rrbracket \to \calE\calD(\llbracket X \rrbracket)$. By naturality of $\eta$, this is equal to $\eta(\llbracket X \rrbracket)\circ x_d$.
\item
Finally, $\llbracket\vdash \lambda.x:X.\lambda \varphi:2^X.\varphi(x):[X\to 2^{2^X}]::\Sort\rrbracket$ maps the unique element of the singleton space to the morphism $\eta(\llbracket X \rrbracket)$.
\end{itemize}
Similarly, when $\Sort=\CHQPol$ the term is interpreted as $\epsilon(\llbracket X \rrbracket)$.


\subsubsection{Sample derivation $(\mu_X\colon 2^4X\to 2^{2^X})$}

\vspace{0.5cm}

\makebox[\linewidth][c]{
\AxiomC{}
\RightLabel{(var)}
\UnaryInfC{$F:2^{2^X}::\Sort \vdash F:2^{2^X}::\Sort$}
\RightLabel{(dual-mor)}
\UnaryInfC{$F:2^{2^X}::\Sort \vdash F:[2^X \to 2]::\dualSort$}
\AxiomC{}
\RightLabel{(var)}
\UnaryInfC{$\varphi:2^X::\dualSort \vdash \varphi:2^X::\dualSort$}
\RightLabel{(app')}
\BinaryInfC{$\varphi:2^X::\dualSort,F:2^{2^X}::\Sort \vdash F(\varphi):2::\dualSort$}
\RightLabel{(dual-obj)}
\UnaryInfC{$\varphi:2^X::\dualSort,F:2^{2^X}::\Sort \vdash F(\varphi):2::\Sort$}
\RightLabel{(curry)}
\UnaryInfC{$\varphi:2^X::\dualSort \vdash \lambda F:2^{2^X}.F(\varphi):[2^{2^X} \to 2]::\Sort$}
\RightLabel{(dual-mor)}
\UnaryInfC{$\varphi:2^X::\dualSort \vdash \lambda F:2^{2^X}.F(\varphi):2^3 X::\dualSort$}
\UnaryInfC{$\vdots$}
\DisplayProof
}

\makebox[\linewidth][c]{
\AxiomC{}
\RightLabel{(var)}
\UnaryInfC{$\IF:2^4 X ::\Sort \vdash \IF:2^4 X ::\Sort$}
\RightLabel{(dual-mor)}
\UnaryInfC{$\IF:2^4 X ::\Sort \vdash \calF:[2^3 X \to 2]::\dualSort$}
\AxiomC{$\vdots$}
\UnaryInfC{$\varphi:2^X::\dualSort \vdash \lambda F:2^{2^X}.F(\varphi):2^3 X::\dualSort$}
\RightLabel{(app')}
\BinaryInfC{$\IF:2^4 X::\Sort, \varphi:2^X::\dualSort \vdash \IF(\lambda F:2^{2^X}.F(\varphi)):2::\dualSort$}
\RightLabel{(curry)}
\UnaryInfC{$\IF:2^4 X::\Sort\vdash \lambda \varphi:2^X.\IF(\lambda F:2^{2^X}.F(\varphi)):[2^X \to 2]::\dualSort$}
\RightLabel{(dual-mor)}
\UnaryInfC{$\IF:2^4 X::\Sort\vdash \lambda \varphi:2^X.\IF(\lambda F:2^{2^X}.F(\varphi)):2^{2^X}::\Sort$}
\RightLabel{(curry)}
\UnaryInfC{$\vdash \lambda \IF:2^4 X. \lambda \varphi:2^X.\IF(\lambda F:2^{2^X}.F(\varphi)): [2^4 X \to 2^{2^X}]::\Sort$}
\DisplayProof
}

\vspace{0.5cm}

In the above derivation, $2^3 X $ is an abbreviation for $2^{2^{2^X}}$, and similarly $2^4 X$ is an abbreviation for $X$ on top of a tower of four $2$'s.

Using the lambda calculus it is easy to verify that $\mu_X = 2^{\eta_{(2^X)}}$ (see Section~3 of \cite{Taylor02}), and the derivation above uses this fact. The first half of the derivation (up until the break) is the derivation of $\eta_{(2^X)}$, and the second half is the derivation of $2^f$ for $f=\eta_{(2^X)}$ (minus the redundant (curry) and (app') where the derivations combine). 

Therefore, when $\Sort = \ODQPol$ the morphism $\eta_{(2^X)}$ is in $\CHQPol$, hence the term $\mu_X$ corresponds to the morphism $\calE(\epsilon(\llbracket 2^X\rrbracket))=\calE(\epsilon(\calD(\llbracket X\rrbracket)))$.  Similarly, when $\Sort = \CHQPol$ the term $\mu_X$ corresponds to the morphism $\calD(\eta(\calE(\llbracket X\rrbracket)))$.

It follows that the interpretation of $(2^{2^X},\eta_X,\mu_X)$ is the corresponding monad from Lemma~\ref{lem:monads}, independent of whether $\Sort$ is $\ODQPol$ or $\CHQPol$.

\subsubsection{Sample derivation $(\sigma_{X,Y}\colon X\times 2^{2^Y}\to 2^{2^{(X\times Y)}})$}

Next, we derive the strength for the monad (see Section~3 of \cite{Taylor02}). We leave the interpretation to the reader.

\vspace{0.5cm}

\makebox[\linewidth][c]{
\small{
\AxiomC{}
\RightLabel{(var)}
\UnaryInfC{$e:X\times 2^{2^Y}::\Sort,y:Y::\Sort\vdash e:X\times 2^{2^Y}::\Sort$}
\RightLabel{(proj)}
\UnaryInfC{$e:X\times 2^{2^Y}::\Sort,y:Y::\Sort\vdash \pi_1(e) :X::\Sort$}
\AxiomC{}
\RightLabel{(var)}
\UnaryInfC{$e:X\times 2^{2^Y}::\Sort,y:Y::\Sort\vdash y:Y::\Sort$}
\RightLabel{(pair)}
\BinaryInfC{$e:X\times 2^{2^Y}::\Sort,y:Y::\Sort\vdash \langle \pi_1(e),y\rangle :X\times Y::\Sort$}
\UnaryInfC{$\vdots$}

\DisplayProof
}
}

\vspace{0.5cm}

\noindent
\makebox[\linewidth][c]{
\small{
\AxiomC{}
\RightLabel{(var)}
\UnaryInfC{$\psi:2^{(X\times Y)}::\dualSort\vdash \psi: 2^{(X\times Y)}::\dualSort$}
\RightLabel{(dual-mor)}
\UnaryInfC{$\psi:2^{(X\times Y)}::\dualSort\vdash \psi: [(X\times Y) \to 2]::\Sort$}

\AxiomC{$\vdots$}
\UnaryInfC{$e:X\times 2^{2^Y}::\Sort,y:Y::\Sort\vdash \langle \pi_1(e),y\rangle :X\times Y::\Sort$}
\RightLabel{(app')}
\BinaryInfC{$e:X\times 2^{2^Y}::\Sort,\psi:2^{(X\times Y)}::\dualSort,y:Y::\Sort \vdash \psi(\pi_1(e),y):2::\Sort$}

\RightLabel{(curry)}
\UnaryInfC{$e:X\times 2^{2^Y}::\Sort,\psi:2^{(X\times Y)}::\dualSort\vdash \lambda y:Y.\psi(\pi_1(e),y):[Y \to 2]::\Sort$}
\RightLabel{(dual-mor)}
\UnaryInfC{$e:X\times 2^{2^Y}::\Sort,\psi:2^{(X\times Y)}::\dualSort\vdash \lambda y:Y.\psi(\pi_1(e),y):2^Y::\dualSort$}
\UnaryInfC{$\vdots$}
\DisplayProof
}
}

\vspace{0.5cm}

\noindent
\makebox[\linewidth][c]{
\small{
\AxiomC{}
\RightLabel{(var)}
\UnaryInfC{$e:X\times 2^{2^Y}::\Sort\vdash e:X\times 2^{2^Y}::\Sort$}
\RightLabel{(proj)}
\UnaryInfC{$e:X\times 2^{2^Y}::\Sort\vdash \pi_2(e):2^{2^Y}::\Sort$}
\RightLabel{(dual-mor)}
\UnaryInfC{$e:X\times 2^{2^Y}::\Sort\vdash \pi_2(e):[{2^Y}\to 2]::\dualSort$}

\AxiomC{$\vdots$}
\UnaryInfC{$e:X\times 2^{2^Y}::\Sort,\psi:2^{(X\times Y)}::\dualSort\vdash \lambda y:Y.\psi(\pi_1(e),y):2^Y::\dualSort$}

\RightLabel{(app')}
\BinaryInfC{$e:X\times 2^{2^Y}::\Sort,\psi:2^{(X\times Y)}::\dualSort\vdash \pi_2(e)(\lambda y:Y.\psi(\pi_1(e),y)):2::\dualSort$}
\RightLabel{(curry)}
\UnaryInfC{$e:X\times 2^{2^Y}::\Sort\vdash \lambda \psi:2^{(X\times Y)}.\pi_2(e)(\lambda y:Y.\psi(\pi_1(e),y)):[2^{(X\times Y)}\to 2]::\dualSort$}

\RightLabel{(dual-mor)}
\UnaryInfC{$e:X\times 2^{2^Y}::\Sort\vdash \lambda \psi:2^{(X\times Y)}.\pi_2(e)(\lambda y:Y.\psi(\pi_1(e),y)):2^{2^{(X\times Y)}}::\Sort$}

\RightLabel{(curry)}
\UnaryInfC{$\vdash \lambda e:X\times 2^{2^Y}. \lambda \psi:2^{(X\times Y)}.\pi_2(e)(\lambda y:Y.\psi(\pi_1(e),y)):[X\times 2^{2^Y} \to 2^{2^{(X\times Y)}}]::\Sort$}
\DisplayProof
}
}

\vspace{0.5cm}

\subsubsection{Sample derivation $\hat{\wedge}\colon A\times A \to 2^{2^A}$}

Finally, we derive some terms that will be used when showing the equivalence of monad algebras and Boolean algebras in Section~\ref{sec:algebras_equiv}. We leave the interpretation to the reader. First,

\vspace{0.5cm}

\noindent
\makebox[\linewidth][c]{
\AxiomC{}
\RightLabel{(var)}
\UnaryInfC{$\varphi:2^A::\dualSort\vdash \varphi :2^A::\dualSort$}
\RightLabel{(dual-mor)}
\UnaryInfC{$\varphi:2^A::\dualSort\vdash \varphi :[A \to 2]::\Sort$}
\AxiomC{}
\RightLabel{(var)}
\UnaryInfC{$e:A\times A::\Sort \vdash e:A\times A::\Sort$}
\RightLabel{(proj)}
\UnaryInfC{$e:A\times A::\Sort \vdash \pi_1(e): A::\Sort$}
\RightLabel{(app')}
\BinaryInfC{$e:A\times A::\Sort,\varphi:2^A::\dualSort\vdash \varphi(\pi_1(a)) :2::\Sort$}
\DisplayProof
}

\vspace{0.5cm}

Similarly, derive $e:A\times A::\Sort,\varphi:2^A::\dualSort\vdash \varphi(\pi_2(a)) :2::\Sort$, and combine the two:

\vspace{0.5cm}

\noindent
\makebox[\linewidth][c]{
\AxiomC{$\vdots$}
\UnaryInfC{$e:A\times A::\Sort,\varphi:2^A::\dualSort\vdash \varphi(\pi_1(a)) :2::\Sort$}
\AxiomC{$\vdots$}
\UnaryInfC{$e:A\times A::\Sort,\varphi:2^A::\dualSort\vdash \varphi(\pi_2(a)) :2::\Sort$}
\RightLabel{(pair)}
\BinaryInfC{$e:A\times A::\Sort,\varphi:2^A::\dualSort\vdash \langle \varphi(\pi_1(a)),\varphi(\pi_2(a))\rangle :2\times 2::\Sort$}
\DisplayProof
}

\vspace{0.5cm}

Finally, apply $\wedge$ and curry.

\vspace{0.5cm}

\noindent
\makebox[\linewidth][c]{
\AxiomC{($\wedge:2\times 2\to  2$ in $\Sort_\Mor$)}
\RightLabel{(const-mor)}
\UnaryInfC{$\vdash \wedge: [2\times 2\to  2]::\Sort$}

\AxiomC{$\vdots$}
\UnaryInfC{$e:A\times A::\Sort,\varphi:2^A::\dualSort\vdash \langle \varphi(\pi_1(a)),\varphi(\pi_2(a))\rangle :2\times 2::\Sort$}

\RightLabel{(app')}
\BinaryInfC{$e:A\times A::\Sort,\varphi:2^A::\dualSort\vdash \varphi(\pi_1(a)) \wedge \varphi(\pi_2(a)):2::\Sort $}

\RightLabel{(dual-obj)}
\UnaryInfC{$e:A\times A::\Sort,\varphi:2^A::\dualSort\vdash \varphi(\pi_1(a)) \wedge \varphi(\pi_2(a)):2::\dualSort $}

\RightLabel{(curry)}
\UnaryInfC{$e:A\times A::\Sort\vdash \lambda \varphi:2^A.\varphi(\pi_1(a)) \wedge \varphi(\pi_2(a)):[2^A \to 2]::\dualSort $}

\RightLabel{(dual-mor)}
\UnaryInfC{$e:A\times A::\Sort\vdash \lambda \varphi:2^A.\varphi(\pi_1(a)) \wedge \varphi(\pi_2(a)): 2^{2^A}::\Sort $}

\RightLabel{(curry)}
\UnaryInfC{$\vdash \lambda e:A\times A. \lambda \varphi:2^A.\varphi(\pi_1(a)) \wedge \varphi(\pi_2(a)): [A\times A \to 2^{2^A}]::\Sort $}
\DisplayProof
}

\vspace{0.5cm}


\section{Equivalence between monad algebras and Boolean algebras}\label{sec:algebras_equiv}

As we saw in the previous section, the monad $(\calE\circ \calD, \eta,  \calE\epsilon\calD)$ on $\ODQPol$, the monad $(\calD\circ \calE, \epsilon, \calD\eta\calE)$ on $\CHQPol$, and the restriction of the latter monad to $\ZCHQPol$, all correspond to the monad $(2^{2^{(\cdot)}},\eta,\mu)$ defined by the lambda terms:
\begin{eqnarray*}
\eta_X &=& \lambda x:X.\lambda f:2^X.f(x):[X\to 2^{2^X}]\\
\mu_X &=& \lambda \IF:2^4 X. \lambda f:2^X.\IF(\lambda F:2^{2^X}.F(f)): [2^4 X \to 2^{2^X}]
\end{eqnarray*}

We briefly recall the Eilenberg-Moore categories of the monads (see Definition~\ref{def:monadalgebra}). The objects (which we call \emph{monad algebras}) consist of a pair $(A,\alpha)$, with $\alpha\colon 2^{2^A}\to A$, such that $\alpha \circ \eta_A$ is the identity function on $A$, and $\alpha \circ \mu_A = \alpha \circ 2^{2^\alpha}$ . By a direct calculation, this requirement translates into the following two equations:
\begin{itemize}
\item
$\alpha(\lambda f.f(a)) = a$ for all $a\in A$, and
\item
$\alpha(\lambda f.\IF(\lambda F.F(f))) = \alpha(\lambda f.\IF(\lambda F.f(\alpha(F))))$ for all $\IF \in 2^4(A)$.
\end{itemize}
An \emph{algebra homomorphism} $h\colon (A,\alpha) \to (B,\beta)$ is a morphism $h\colon A\to B$ such that $h\circ \alpha = \beta\circ 2^{2^h}$. This translates into the equation:
\begin{itemize}
\item
$h(\alpha(F)) = \beta(\lambda g. F(\lambda a.g(h(a))))$ for all $F \in 2^{2^A}$.
\end{itemize}

In this section, we show that these monad algebras are precisely the Boolean algebras, in the respective categories. Since $\Alg(\CHQPol)=\Alg(\ZCHQPol)$ (Lemma~\ref{lem:algch=algzch}) and $\BA(\CHQPol)=\BA(\ZCHQPol)$ (Lemma~\ref{lem:bach=bazch}), we only state the results in this section for $\ZCHQPol$. For notational simplicity, we do not differentiate between the monad on $\CHQPol$ and its restriction to $\ZCHQPol$. This is not problematic because of Lemma~\ref{lem:DX_is_zerodim}. 

We will use the following notational convention throughout this section: $a,b,c: A$, $f,g: 2^A$, $F: 2^{2^A}$, $\calF: 2^3(A)$, $\IF: 2^4(A)$.

\subsection{From monad algebras to Boolean algebras}

We define computable functors $\calF_{\BA} \colon \Alg(\ODQPol) \to \BA(\ODQPol)$ and $\calG_{\BA} \colon \Alg(\ZCHQPol) \to \BA(\ZCHQPol)$ as follows. The construction is identical in both cases, so we present them together. Each algebra $(A,\alpha)$ is mapped to a Boolean algebra object $(A,\vee_\alpha,\wedge_\alpha,\neg_\alpha,0_\alpha,1_\alpha)$ by lifting the structure from the two element Boolean algebra $(2,\vee,\wedge,\neg,0,1)$ as follows:
\begin{itemize}
\item
$\vee_\alpha \colon A\times A\to A$ is defined as  $a \vee_\alpha b = \alpha(\lambda f.f(a)\vee f(b))$,
\item
$\wedge_\alpha \colon A\times A\to A$ is defined as  $a \wedge_\alpha b = \alpha(\lambda f.f(a)\wedge f(b))$,
\item
$\neg_\alpha \colon A\to A$ is defined as  $\neg_\alpha a = \alpha(\lambda f.\neg f(a))$,
\item
$0_\alpha \colon 1 \to A$ is defined as  $0_\alpha = \alpha(\lambda f. 0)$,
\item
$1_\alpha \colon 1 \to A$ is defined as $1_\alpha = \alpha(\lambda f.1)$.
\end{itemize}
Each algebra homomorphism $h\colon (A,\alpha) \to (B,\beta)$ is mapped to the  Boolean algebra homomorphism that has the same underlying morphism $h\colon A\to B$.

Given a Boolean algebra term $\varphi$ in the language of $(A,\vee_\alpha,\wedge_\alpha,\neg_\alpha,0_\alpha,1_\alpha)$  and a variable symbol $f$ of type $2^A$, define the term $\varphi_f$ as follows:
\begin{enumerate}
\item
If $\varphi = a$, where $a \in A$, then $\varphi_f = f(a)$.
\item
If $\varphi = \psi \vee_\alpha \theta$ then $\varphi_f = (\psi_f \vee \theta_f)$.
\item
If $\varphi = \psi \wedge_\alpha \theta$ then $\varphi_f = (\psi_f \wedge \theta_f)$.
\item
If $\varphi = \neg_\alpha \psi$ then $\varphi_f = (\neg \psi_f)$.
\item
If $\varphi = 0_\alpha$ then $\varphi_f = 0$. 
\item
If $\varphi = 1_\alpha$ then $\varphi_f = 1$. 
\end{enumerate}
Note that since all occurrences of $f$ in $\varphi_f$ are unbounded, if $g\colon 2^A$ is any other variable symbol then $(\lambda f.\varphi_f)(g) = \varphi_g$.

\begin{lemma}\label{lem:boolterm_regform}
Let $(A,\alpha)$ be an object of $\Alg(\ODQPol)$ or $\Alg(\ZCHQPol)$. If $\varphi$ is a Boolean algebra term in the language of $(A,\vee_\alpha,\wedge_\alpha,\neg_\alpha,0_\alpha,1_\alpha)$, then $\varphi = \alpha(\lambda f. \varphi_f)$ when evaluated in $A$.
\end{lemma}
\begin{proof}
By induction on the structure of $\varphi$. The cases for $\varphi = 0_\alpha$ and $\varphi = 1_\alpha$ hold by definition. The case for $\varphi = a$ holds because $\alpha(\lambda f.f(a)) = a$ by definition of a structure map.

Assume $\varphi = \psi \vee_\alpha \theta $ and $\psi = \alpha(\lambda f. \psi_f)$ and $\theta = \alpha(\lambda f. \theta_f)$. Define
\[\IF = \lambda \calF.\calF(\lambda f.\psi_f) \vee \calF(\lambda f. \theta_f).\]
Then
\begin{eqnarray*}
\varphi &=& \psi \vee_\alpha \theta \\
&=& \alpha(\lambda g. g(\psi) \vee g(\theta))\\
&=& \alpha(\lambda g. g(\alpha(\lambda f. \psi_f)) \vee g(\alpha(\lambda f. \theta_f)))\\
&=& \alpha(\lambda g. \IF(\lambda F.g(\alpha(F))))\\
&=& \alpha(\lambda g. \IF(\lambda F.F(g)))\\
&=& \alpha(\lambda g. (\lambda f.\psi_f)(g) \vee (\lambda f.\theta_f)(g))\\
&=& \alpha(\lambda g. \psi_g \vee \theta_g)\\
&=& \alpha(\lambda g. \varphi_g),
\end{eqnarray*}
hence a change of variables yields $\varphi =  \alpha(\lambda f. \varphi_f)$. The case for $\varphi = \psi \wedge_\alpha \theta$ is essentially the same, and the case for $\varphi = \neg \psi$ is obtained using $\IF = \lambda \calF.\neg \calF(\lambda f.\psi_f)$.
\qed
\end{proof}

\begin{lemma}
The functors $\calF_{\BA} \colon \Alg(\ODQPol) \to \BA(\ODQPol)$ and $\calG_{\BA} \colon \Alg(\ZCHQPol) \to \BA(\ZCHQPol)$ are well-defined and computable.
\end{lemma}
\begin{proof}
Let $(A,\alpha)$ be an object of either $\Alg(\ODQPol)$ or $\Alg(\ZCHQPol)$. We first check that $(A,\vee_\alpha,\wedge_\alpha,\neg_\alpha,0_\alpha,1_\alpha)$ is a Boolean algebra object. This is done by using Lemma~\ref{lem:boolterm_regform} to show that the corresponding axioms for $2$ lift to the whole algebra. For example, the distributivity axiom holds because
\begin{eqnarray*}
a \wedge_\alpha (b \vee_\alpha c) &=&  \alpha(\lambda g. g(a)\wedge (g(b)\vee g(c)))\\
&=& \alpha(\lambda g.(g(a)\wedge g(b))\vee (g(a) \wedge g(c))) \\
&=& (a \wedge_\alpha b)\vee_\alpha (a \wedge_\alpha c),
\end{eqnarray*}
where the first and third equalities are Lemma~\ref{lem:boolterm_regform}, and the second equality follows from the distributivity axiom for  $(2,\vee,\wedge,\neg,0,1)$. The other axioms can be handled similarly. It follows that $(A,\vee_\alpha,\wedge_\alpha,\neg_\alpha,0_\alpha,1_\alpha)$ is a Boolean algebra object, and it is clear that it can be computed given $(A,\alpha)$.

Next we show that algebra homomorphisms are Boolean algebra homomorphisms. Assume $(A,\alpha)$ and $(B,\beta)$ are algebras and $h\colon A\to B$ is an algebra homomorphism. One can directly verify that $h$ preserves the Boolean algebra structure. For example, the preservation of joins is shown as follows:
\begin{eqnarray*}
h(a \vee_\alpha b) &=& h(\alpha(\lambda f.f(a) \vee f(b)))\\
&=& \beta(2^{2^h}(\lambda f.f(a) \vee f(b)))\\
&=& \beta(\lambda g. (\lambda f.f(a) \vee f(b))(\lambda x.g(h(x))))\\
&=& \beta(\lambda g. g(h(a)) \vee g(h(b)))\\
&=& h(a) \vee_\beta h(b).
\end{eqnarray*}
The other cases are similar and left to the reader. 
\qed
\end{proof}

\begin{lemma}\label{lem:induced_ultrafilters_are_points}
If $(A, \alpha)$ is an algebra (of either $\Alg(\ODQPol)$ or $\Alg(\ZCHQPol)$), and $h\colon A\to 2$ is a (continuous) Boolean algebra homomorphism with respect to $(A,\vee_\alpha,\wedge_\alpha,\neg_\alpha,0_\alpha,1_\alpha)$, and $F\colon 2^A\to 2$ is a morphism, then $h(\alpha(F))=F(h)$.
\end{lemma}
\begin{proof}
Let $h\colon A\to 2$ be a Boolean algebra homomorphism and $F\colon 2^A\to 2$. 

We first consider the case where $A$ is in $\ODQPol$. If $\varphi$ is a Boolean algebra term in the language of $(A,\vee_\alpha,\wedge_\alpha,\neg_\alpha,0_\alpha,1_\alpha)$, then $h(\varphi)= \varphi_h$ by an easy induction. Since $F$ is continuous, there exist basic clopen subsets $C_1,\ldots, C_n$ of $2^A$ such that $F(g)=1$ if and only if $g\in \bigcup_{i\leq n} C_i$. Using the fact that $A$ is discrete and $2^A$ has the compact-open topology, we can assume
\[C_i = \{ g\in 2^A \mid (\forall a\in F^0_i)\, g(a)=0 \text{ and } (\forall a\in F^1_i)\, g(a)=1\}\]
for some finite $F^0_i,F^1_i \subseteq A$. Set
\[\varphi=\bigvee_{1\leq i \leq n} \left(\bigwedge_{a\in F^0_i} \neg a \wedge \bigwedge_{a\in F^1_i} a \right),\]
where $\bigvee$ and $\bigwedge$ are the Boolean operations in $(A,\vee_\alpha,\wedge_\alpha,\neg_\alpha,0_\alpha,1_\alpha)$. Then $F = \lambda f.\varphi_f$, hence using Lemma~\ref{lem:boolterm_regform} we obtain $h(\alpha(F)) = h(\alpha(\lambda f.\varphi_f)) = h(\varphi) = \varphi_h = F(h)$.

Next, we consider the case where $A$ is in $\ZCHQPol$. The proof is similar to the case for $\ODQPol$, but we must go into more detail about the construction of the term $\varphi$. From continuity of $h\circ\alpha$ and the fact that $2^A$ is discrete, there are $f_0,\ldots, f_n\in 2^A$ such that if $G\colon 2^A\to 2$ and $G(f_i)=F(f_i)$ for all $i\leq n$ then $h(\alpha(G))=h(\alpha(F))$. We assume $f_0=h$ and that they are pairwise distinct, so there is $a_{i,j}$ ($i,j\leq n$) such that $f_i(a_{i,j}) \not= f_j(a_{i,j})$ whenever $i\not=j$. For each $i\leq n$, let $\theta_i$ be the $(A,\vee_\alpha,\wedge_\alpha,\neg_\alpha,0_\alpha,1_\alpha)$-Boolean algebra term 
\[ \theta_i = (\widehat{a_{i,0}} \wedge_{\alpha} \cdots \wedge_{\alpha} \widehat{a_{i,n}})\wedge_{\alpha} (F(f_i))_{\alpha}\]
where
\[\widehat{a_{i,j}} = \left\{\begin{array}{ll} a_{i,j} & \text{ if $f_i(a_{i,j})=1$}\\ \neg_{\alpha} a_{i,j} & \text{ if $f_i(a_{i,j})=0$} \end{array}\right.\]
and
\[ (F(f_i))_{\alpha} = \left\{\begin{array}{ll} 1_\alpha & \text{ if $F(f_i)=1$}\\ 0_{\alpha} & \text{ if $F(f_i)=0$} \end{array}\right.\]
Finally, set
\[\varphi = \theta_0 \vee_\alpha \cdots \vee_\alpha \theta_n.\]

Then $(\lambda g.\varphi_g)(f_i) = F(f_i)$ for each $i\leq n$, because the terms corresponding to $\theta_j$ will evaluate to $0$ for $j\not= i$, and $\theta_i$ will evaluate to $F(f_i)$. Therefore, $h(\alpha(F)) = h(\alpha( \lambda g.\varphi_g))$. Using Lemma~\ref{lem:boolterm_regform} and the assumption that $h$ is a Boolean algebra homomorphism, we have
\begin{eqnarray*}
h(\alpha(F)) &=& h(\alpha( \lambda g.\varphi_g))\\
&=& h(\varphi)\\
&=& \varphi_h\\
&=& (\lambda g.\varphi_g)(h)\\
&=& F(h),
\end{eqnarray*}
where the last equality holds because $h=f_0$.
\qed
\end{proof}

\subsection{From Boolean algebras to monad algebras}

The key lemma is the following.

\begin{lemma}\label{lem:boolhoms_are_points}
Given a Boolean algebra object $(A, \vee, \wedge,\neg,0,1)$ in $\ODQPol$ or $\ZCHQPol$, one can compute a morphism $\alpha\colon 2^{2^A}\to A$ such that $h(\alpha(F))=F(h)$ for every continuous Boolean algebra homomorphism $h\colon A\to 2$ and continuous function $F\colon 2^A\to 2$.
\qed
\end{lemma}

We will prove the lemma separately for $\ODQPol$ and $\ZCHQPol$ later, but first we show how to use the above lemma to define computable functors $\calF_{\Alg} \colon \BA(\ODQPol)\to \Alg(\ODQPol)$ and $\calG_{\Alg} \colon \BA(\ZCHQPol)\to \Alg(\ZCHQPol)$. The definition of the functors is the same for both $\ODQPol$ and $\ZCHQPol$: Each Boolean algebra object $(A, \vee, \wedge,\neg,0,1)$ is mapped to $(A,\alpha)$, where $\alpha\colon 2^{2^A}\to A$ is as in Lemma~\ref{lem:boolhoms_are_points}. Each Boolean algebra homomorphism $h\colon A\to B$ is mapped to itself. 

\begin{lemma}
The functors  $\calF_{\Alg} \colon \BA(\ODQPol)\to \Alg(\ODQPol)$ and $\calG_{\Alg} \colon \BA(\ZCHQPol)\to \Alg(\ZCHQPol)$ are well-defined and computable.
\end{lemma}
\begin{proof}
The functors are computable because $\alpha$ can be computed from $A$, so we only need to show they are well-defined.

We first show that the $\alpha$ in Lemma~\ref{lem:boolhoms_are_points} is uniquely determined. Assume $\alpha'\colon 2^{2^A}\to A$ also satisfies the lemma. Then $h(\alpha(F))=F(h)=h(\alpha'(F))$ for each continuous $F\colon 2^A\to 2$ and continuous Boolean algebra homomorphism $h\colon A\to 2$. Therefore, $\alpha(F)=\alpha'(F)$ because they can not be separated by a continuous Boolean algebra homomorphism.

Next, to show that $\alpha$ is a structure map we must prove it satisfies the following:
\begin{enumerate}
\item
$\alpha(\lambda f.f(a)) = a$,
\item
$\alpha(\lambda f.\IF(\lambda F.F(f))) = \alpha(\lambda f.\IF(\lambda F.f(\alpha(F))))$.
\end{enumerate}
For the first equation, if $h\colon A\to 2$ is a continuous Boolean algebra homomorphism, then the assumption on $\alpha$ implies $h(\alpha(\lambda f.f(a))) = (\lambda f.f(a))(h) = h(a)$. Therefore, $\alpha(\lambda f.f(a)) = a$ because they can not be separated by a continuous Boolean algebra homomorphism.  Similarly, the second equation holds because if $h\colon A\to 2$ is a continuous Boolean algebra homomorphism then
\begin{eqnarray*}
 h(\alpha(\lambda f.\IF(\lambda F.F(f)))) &=& (\lambda f.\IF(\lambda F.F(f)))(h)\\
&=& \IF(\lambda F.F(h))\\
&=& \IF(\lambda F.h(\alpha(F)))\\
&=& (\lambda f.\IF(\lambda F.f(\alpha(F))))(h)\\
&=& h(\alpha(\lambda f.\IF(\lambda F.f(\alpha(F))))).
\end{eqnarray*}
Therefore, $\alpha$ is a structure map.

Finally, assume $h\colon A\to B$ is a Boolean algebra homomorphism. To prove that $h\colon (A,\alpha)\to (B,\beta)$ is an algebra homomorphism, we must show that $h\circ \alpha = \beta \circ 2^{2^h}$, where $2^{2^h}\colon 2^{2^A}\to 2^{2^B}$ is defined as $2^{2^h}(F)=\lambda g:2^B. F(\lambda a.g(h(a)))$.  

For $F\colon 2^{2^A}$ and any continuous Boolean algebra homomorphism $p\colon B\to 2$, we have
\begin{eqnarray*}
p(h(\alpha(F)) ) &=& F(p\circ h)\\
&=&F(\lambda a.p(h(a)))  \\
&=&(\lambda g. F(\lambda a.g(h(a))))(p)   \\
&=& p( \beta(\lambda g. F(\lambda a.g(h(a))))),
\end{eqnarray*}
hence $h(\alpha(F)) = \beta(\lambda g. F(\lambda a.g(h(a))))$. Therefore, $h\circ \alpha = \beta \circ 2^{2^h}$.
\qed
\end{proof}

We now move on to the proof of Lemma~\ref{lem:boolhoms_are_points}.

\subsubsection*{Proof of Lemma~\ref{lem:boolhoms_are_points} for $\ODQPol$}

Let $(A, \vee, \wedge,\neg,0,1)$ be a Boolean algebra object in $\ODQPol$. Define $\alpha\colon 2^{2^A} \to A$ as follows. Given $F\colon 2^A\to 2$, the continuity of $F$ and the compactness of $2^A$ implies there is  a sequence $\langle C_0, D_0\rangle, \ldots, \langle C_n, D_n\rangle$ of finite sets $C_i,D_i\subseteq A$ such that
\begin{equation}\label{eqn:repset}
F(f)=  \bigvee_{0\leq i \leq n}\left( \bigwedge_{c\in C_i} f(c) \wedge \bigwedge_{c\in D_i} \neg f(c) \right)\end{equation}
for each $f\colon A\to 2$, where the Boolean algebra operations are with respect to $(2,\vee,\wedge,\neg,0,1)$. Fix such a sequence and define
\[\alpha(F)= \bigvee_{0\leq i \leq n}\left( \bigwedge_{c\in C_i} c \wedge \bigwedge_{c\in D_i} \neg c \right),\]
where the Boolean algebra operations are with respect to $(A, \vee, \wedge,\neg,0,1)$.

If $h\colon A\to 2$ is a Boolean algebra homomorphism, then
\begin{eqnarray*}
h(\alpha(F)) &=& \bigvee_{0\leq i \leq n}\left( \bigwedge_{c\in A_i} h(c) \wedge \bigwedge_{c\in B_i} \neg h(c) \right)\\
&=& F(h).
\end{eqnarray*}

The definition of $\alpha$ is independent of the particular sequence $(\langle C_i, D_i\rangle)_{0\leq i\leq n}$  chosen as long as it satisfies (\ref{eqn:repset}), because if $(\langle C'_i, D'_i\rangle)_{0\leq i\leq m}$ is another such sequence and we set
\[\alpha'(F)= \bigvee_{0\leq i \leq m}\left( \bigwedge_{c\in C'_i} c \wedge \bigwedge_{c\in D'_i} \neg c \right),\]
then $h(\alpha(F))=F(h) = h(\alpha'(F))$ for every Boolean algebra homomorphism $h\colon A\to 2$, hence Lemma~\ref{lem:ods_enuf_homs} implies $\alpha(F)=\alpha'(F)$. 

A sequence $(\langle C_i, D_i\rangle)_{0\leq i\leq n}$  satisfying (\ref{eqn:repset}) can be found by an exhaustive search, because the finite subsets of $A$ can be enumerated, and verifying that (\ref{eqn:repset}) holds for all $f\colon 2\to A$ is semi-decidable by the compactness of $2^A$. It follows that $\alpha$ is computable. 

Therefore, Lemma~\ref{lem:boolhoms_are_points} holds for $\ODQPol$.

\subsubsection*{Proof of Lemma~\ref{lem:boolhoms_are_points} for $\ZCHQPol$}

Define $\odot\colon 2\times A\to A$ as $0\odot b = \neg b$ and $1 \odot b = b$. Given a Boolean algebra $(A, \vee, \wedge,\neg,0,1)$ in $\ZCHQPol$, define $\alpha\colon 2^{2^A}\to A$ as follows (for $F \colon 2^A\to 2$):
\[\alpha(F) = \bigvee_{\substack{f \in 2^A,\\ F(f)=1}}\bigwedge_{a \in A} (f(a) \odot a).\]
The meet is well-defined by Lemma~\ref{lem:zch_alg_complete}. The join is well-defined by Lemma~\ref{lem:zch_alg_seq_complete}, because $2^A$ is enumerable.

Fix $f\in 2^A$ and set
\[ b_f =  \bigwedge_{a\in A} (f(a)\odot a). \]
If $f(0)=1$ then $b_f \leq f(0)\odot 0 = 0$, and if $f(1)=0$ then $b_f \leq f(1)\odot 1 = 0$. If $f(a)=f(\neg a)$ then $b_f \leq a \wedge \neg a = 0$. If $f(a)=1$ and $f(b)=1$ and $f(a\wedge b)=0$, then $b_f \leq a\wedge b\wedge \neg(a\wedge b) = 0$. If $f(a)=0$ and $f(a\wedge b)=1$ then $b_f \leq \neg a \wedge (a\wedge b)=0$. If $f(a)=0$ and $f(b)=0$ and $f(a\vee b)=1$, then $b_f \leq \neg a\wedge \neg b\wedge (a\vee b) = 0$. If $f(a)=1$ and $f(a\vee b)=0$ then $b_f \leq a \wedge \neg(a\vee b)=0$. Therefore, if $f$ is not a Boolean algebra homomorphism then $b_f=0$.

Let $h\colon A\to 2$ be a continuous Boolean algebra homomorphism and $F\colon 2^A\to 2$. Set $c = \bigwedge h^{-1}(\{1\})$. Since $h$ is a Boolean algebra homomorphism, $h^{-1}(\{1\})$ is closed under finite meets, hence $h(c)=1$ by Lemma~\ref{lem:zch_alg_complete}. Thus for every $a\in A$, either $c \leq a$ (if $h(a)=1$), or else $a \leq \neg c$ (because $h(a)=0$ implies $h(\neg a)=1$). It follows that $h(\alpha(F)) = 1$ if and only if $c \leq \alpha(F)$ if and only if there is $f\in 2^A$ with $F(f) = 1$ and $c \leq \bigwedge_{a \in A} (f(a) \odot a)$ (because a join can be above $c$ if and only if one of the elements is not below $\neg c$). Since $c\not= 0$, $c \leq \bigwedge_{a \in A} (f(a) \odot a)$ if and only if $f$ is a Boolean algebra homomorphism satisfying $f(c)=1$ (otherwise the meet is below $f(c) \odot c = \neg c$), which holds if and only if $f=h$. Thus $h(\alpha(F))=1$ if and only if $F(h)=1$.

For any $F\colon 2^A\to 2$ and $a\in A$, Lemma~\ref{lem:zchs_enuf_homs} implies $\alpha(F)\not= a$ if and only if there is homomorphism $h$ such that $h(\alpha(F))\not=h(a)$, hence if and only if there is a Boolean algebra homomorphism $h$ such that $F(h)\not=h(a)$. Since the set of Boolean algebra homomorphisms $h\colon A\to 2$ can be enumerated, $\alpha(F)\not= a$ is semi-decideable. It follows that $\alpha$ is computable.

Therefore, Lemma~\ref{lem:boolhoms_are_points} holds for $\ZCHQPol$.

\subsection{Proof of equivalence}

We prove the following in this section.

\begin{theorem}
1. The categories $\Alg(\ODQPol)$ and $\BA(\ODQPol)$ are computably isomorphic via the computable functors $\calF_{\BA} \colon \Alg(\ODQPol) \to \BA(\ODQPol)$ and $\calF_{\Alg} \colon \BA(\ODQPol)\to \Alg(\ODQPol)$.

2. The categories $\Alg(\ZCHQPol)$ and $\BA(\ZCHQPol)$ are computably isomorphic via the computable functors $\calG_{\BA} \colon \Alg(\ZCHQPol) \to \BA(\ZCHQPol)$ and $\calG_{\Alg} \colon \BA(\ZCHQPol)\to \Alg(\ZCHQPol)$.
\end{theorem}
\begin{proof}
The proof is the same for both $\ODQPol$ and $\ZCHQPol$. The functors between the categories do not change the underlying space, so we only need to check that the algebraic structure is unchanged.

First, start with $(A, \vee, \wedge,\neg,0,1)$, and apply $\calF_{\BA}$ (or $\calG_{\BA}$) to get $(A,\alpha)$, and then apply $\calF_{\Alg}$ (or $\calG_{\Alg}$) to get $(A, \vee', \wedge',\neg',0',1')$. Set $F = \lambda f.f(a)\vee f(b)$.  Then $a \vee' b = \alpha(F)$, and if $h\colon A\to 2$ is a Boolean homomorphism with respect to $(A, \vee, \wedge,\neg,0,1)$, then 
\begin{eqnarray*}
h(a\vee'b) &=& h(\alpha(F)) \text{ (definition of $\vee'$)}\\
&=& F(h) \text{ (Lemma~\ref{lem:boolhoms_are_points})}\\
&=& h(a)\vee h(b)\text{ (definition of $F$)}\\
&=& h(a\vee b).
\end{eqnarray*}
Therefore, $a\vee' b = a\vee b$, because they cannot be separated by a Boolean homomorphism. The other operations are handled similarly. Therefore,  $(A, \vee, \wedge,\neg,0,1)=(A, \vee', \wedge',\neg',0',1')$.

Next, start with $(A,\alpha)$,  and apply $\calF_{\BA}$ (or $\calG_{\BA}$) to get $(A, \vee, \wedge,\neg,0,1)$, and then apply $\calF_{\Alg}$ (or $\calG_{\Alg}$) to get $(A,\alpha')$. For any continuous $F\colon 2^A\to 2$, if $h\colon A\to 2$ is a continuous Boolean algebra homomorphism with respect to $(A, \vee, \wedge,\neg,0,1)$, then $h(\alpha(F)) = F(h) = h(\alpha'(F))$ (the first equality by Lemma~\ref {lem:induced_ultrafilters_are_points} and the second by Lemma~\ref{lem:boolhoms_are_points}). Therefore, $\alpha(F)=\alpha'(F)$ because they cannot be separated by a Boolean homomorphism. Therefore, $(A,\alpha)=(A,\alpha')$.
\qed
\end{proof}

\section{Computable Stone dualities}\label{sec:StoneDuality}

\subsection{Sobriety}

Our definition of sobriety is from Definition~4.7 of \cite{Taylor02}. Let $X$ be an object of $\ODQPol$ or $\CHQPol$. We say $X$ is \emph{sober} if the following diagram is an equalizer:
\begin{center}
\begin{tikzcd}
X\arrow[r, "\eta_X"]& 2^{2^{X}}\arrow[r, yshift=0.7ex, "2^{2^{\eta_X}}"]\arrow[r, yshift=-0.7ex, swap, "\eta_{2^{2^X}}"]& 2^{2^{2^{2^{X}}}}
\end{tikzcd}
\end{center}
Using lambda terms, $2^{2^{\eta_X}}\colon 2^{2^X}\to 2^{2^{2^{2^{X}}}}$ is the function 
\[2^{2^{\eta_X}} = \lambda F:2^{2^X}.\lambda \calF:2^{2^{2^X}}.F(\lambda x:X.\calF(\eta_X(x))),\]
and $\eta_{2^{2^X}}\colon 2^{2^X}\to 2^{2^{2^{2^{X}}}}$ is the function 
\[\eta_{2^{2^X}} = \lambda F:2^{2^X}.\lambda \calF:2^{2^{2^X}}.\calF(F).\]
Therefore, $G:2^{2^X}$ is in the equalizer if and only if 
\[\calF(G) = G(\lambda x:X.\calF(\eta_X(x)))\]
for each $\calF:2^{2^{2^X}}$. 

For the special case $G=\eta_X(y)$ for some $y\in 2^X$, we have
\begin{eqnarray*}
\eta_X(y)(\lambda x:X.\calF(\eta_X(x))) &=& (\lambda f.f(y))(\lambda x:X.\calF(\eta_X(x)))\\
&=& (\lambda x:X.\calF(\eta_X(x)))(y)\\
&=& \calF(\eta_X(y)),
\end{eqnarray*}
hence $2^{2^{\eta_X}} \circ \eta_X = \eta_{2^{2^X}}\circ \eta_X$ always holds.

\begin{theorem}
Every object of  $\ODQPol$ is sober. An object of $\CHQPol$ is sober if and only if it is zero-dimensional.
\end{theorem}
\begin{proof}
Assume $X\in\ODQPol_\Obj$. It is clear that $\eta_X$ is injective, hence by Corollary~\ref{cor:bij_iso} it suffices to show that every element in the equalizer of $2^{2^{\eta_X}}$ and $\eta_{2^{2^X}}$ is in the range of $\eta_X$. So assume $G:2^{2^X}$ is such that $\calF(G) = G(\lambda x:X.\calF(\eta_X(x)))$ for each $\calF:2^{2^{2^X}}$. Then 
\begin{eqnarray*}
1&=&(\lambda F:2^{2^X}.1)(G)\\
&=& G(\lambda x:X.(\lambda F:2^{2^X}.1)(\eta_X(x)))\\
&=&G(\lambda x:X.1).
\end{eqnarray*}
Since $G$ is continuous, there is a finite subset $M\subseteq X$ such that $G(f)=1$ whenever $f:2^X$ satisfies $(\forall x\in M)\, f(x)=1$. Define $\calF\colon 2^{2^X}\to 2$ as $\calF(F) = 1$ if and only if $(\exists x\in M)\, F=\eta_X(x)$. Then
\begin{eqnarray*}
\calF(G) &=& G(\lambda x:X.\calF(\eta_X(x)))\\
&=& 1,
\end{eqnarray*}
because $\calF(\eta_X(x))=1$ for each $x\in M$. Therefore, there is $x\in M$ with $G=\eta_X(x)$.

Next, assume $X\in \CHQPol_\Obj$. First assume $X$ is sober. Since $\eta_X$ is an embedding of $X$ into $2^{2^X}$ as a $\lpi 1$-subset, and $2^{2^X}$ is zero-dimensional by Lemma~\ref{lem:DX_is_zerodim}, it follows that $X$ is zero-dimensional.

Conversely, assume $X$ is zero-dimensional. It follows from part (i) of Lemma~\ref{lem:zdim} that $\eta_X(x) = \eta_X(y)$ implies $x=y$, hence $\eta_X$ is injective. Therefore, by Corollary~\ref{cor:bij_iso_ch} it suffices to show that every element in the equalizer of $2^{2^{\eta_X}}$ and $\eta_{2^{2^X}}$ is in the range of $\eta_X$. Let $G:2^{2^X}$ be such that $\calF(G) = G(\lambda x:X.\calF(\eta_X(x)))$ for each $\calF:2^{2^{2^X}}$. Assume for a contradiction that $G \not= \eta_X(x)$ for each $x\in X$. Then for each $x\in X$ there is $f:2^X$ such that $G(f) \not=\eta_X(x)(f) = f(x)$. By compactness of $X$, there is a finite subset $S\subseteq 2^X$ such that for each $x\in X$ there is $f\in S$ with $G(f) \not= f(x)$. Define $\calF:2^{2^X}\to 2$ as 
\[\calF(F)=1 \iff (\exists f\in S)\, G(f) \not=F(f).\]
Then $\calF(G)=0$ and $(\forall x\in X)\, \calF(\eta_X(x))=1$. Therefore,
\begin{eqnarray*}
0 &=& \calF(G)\\
&=& G(\lambda x:X.\calF(\eta_X(x))) \\
&=& G(\lambda x:X.(\lambda F:2^{2^X}. 1)(\eta_X(x))) \\
&=& (\lambda F:2^{2^X}. 1)G \\
&=& 1,
\end{eqnarray*}
which is a contradiction.
\qed
\end{proof}

\subsection{Spatiality}

Fix an object $(A,\alpha)$ of  $\Alg(\ODQPol)$ or $\Alg(\ZCHQPol)$. Define $\pt(A)$ to be the equalizer of the following diagram:
\begin{center}
\begin{tikzcd}
\pt(A)\arrow[r, "e_A"]&2^A \arrow[r, yshift=0.7ex, "2^\alpha"]\arrow[r, yshift=-0.7ex, swap, "\eta_{(2^A)}"]& 2^{2^{2^A}}
\end{tikzcd}
\end{center}
By Corollary~4.4 of \cite{Taylor02}, the elements of $\pt(A)$ are precisely the algebra homomorphisms $A\to 2$.

As lambda terms, we have $2^\alpha = \lambda f:2^A.\lambda F:2^{2^A}.f(\alpha(F))$ and $\eta_{2^A} = \lambda f:2^A.\lambda F:2^{2^A}. F(f)$. Therefore, if $p\in \pt(A)$ and $F\in 2^{2^A}$ then 
\begin{eqnarray*}
e_A(p)(\alpha(F)) &=& 2^\alpha(e_A(p))(F)\\
&=& \eta_{(2^A)}(e_A(p))(F)\\
&=& F(e_A(p)).
\end{eqnarray*}
Conversely, by definition of an equalizer, if $f\in 2^A$ satisfies $f(\alpha(F)) = F(f)$ for all $F\in 2^{2^A}$ then there is $p\in\pt(A)$ with $e_A(p)=f$.

Applying $2^{(-)}$ to the equalizer defining $\pt(A)$, we obtain the top row of the diagram below. 
\begin{center}
\begin{tikzcd}
2^{\pt(A)}&\arrow[l,swap, "2^{(e_A)}"]\arrow[dl,"\alpha"] 2^{2^A}& \arrow[l, yshift=0.7ex, swap,"2^{2^\alpha}"]\arrow[l, yshift=-0.7ex, "2^{\eta_{(2^A)}}"] 2^{2^{2^{2^A}}}\\
A\arrow[u,dashrightarrow,"u"]&&
\end{tikzcd}
\end{center}
Since $(A,\alpha)$ is a monad algebra it is the coequalizer of the parallel pair of arrows (see the proof of Theorem~1 \S VI.7 in \cite{M98}), hence there is a unique $u$ making the diagram above commute. We say $(A,\alpha)$ is \emph{spatial} if $u$ is an algebra isomorphism to $(2^{\pt(A)},2^{\eta_{\pt(A)}})$.

\begin{theorem}\label{thm:spatial}
Every object $(A,\alpha)$ of  $\Alg(\ODQPol)$ or $\Alg(\ZCHQPol)$ is spatial.
\end{theorem}
\begin{proof}
Define $u = 2^{e_A} \circ \eta_A$, or as a lambda term, $u = \lambda a:A.\lambda p:\pt(A).e_A(p)(a)$. Note that $u$ is uniquely determined because if $f$ is such that $2^{e_A} = f\circ \alpha$, then $u = 2^{e_A}\circ \eta_A = f\circ\alpha\circ \eta_A = f$ (also note that the definition of $u$ comes from applying the natural bijection in Theorem~\ref{thrm:adjunction} to $e_A$). 

For $F:2^{2^A}$ we have
\begin{eqnarray*}
u(\alpha(F)) &=& \lambda p:\pt(A).e_A(p)(\alpha(F))\\
&=& \lambda p:\pt(A).F(e_A(p))\\
&=& 2^{e_A}(F),
\end{eqnarray*}
hence $u\circ \alpha = 2^{e_A}$. Also, $u$ is an algebra homomorphism because:
\begin{eqnarray*}
u\circ \alpha &=& 2^{e_A}\\
&=& 2^{e_A}\circ 2^{2^\alpha}\circ 2^{2^{\eta_A}}\text{ ($(A,\alpha)$ is an algebra)}\\
&=& 2^{e_A}\circ 2^{\eta_{(2^A)}}\circ 2^{2^{\eta_A}} \text{ ($2^\alpha\circ e_A = \eta_{(2^A)}\circ e_A$)}\\
&=& 2^{\eta_{\pt(A)}}\circ 2^{2^{2^{e_A}}}\circ 2^{2^{(\eta_A)}}\text{ ( $\eta$ is a natural transformation)}\\
&=& 2^{\eta_{\pt(A)}}\circ 2^{2^u} \text{ (definition of $u$)}.
\end{eqnarray*}

Next we show that $u$ is injective. If $a,b\in A$ are distinct, then Theorem~\ref{thm:enuf_homs} implies there is $h\colon A\to 2$ which is a continuous Boolean algebra homomorphism with respect to $(A,\vee_\alpha,\wedge_\alpha,\neg_\alpha,0_\alpha,1_\alpha)$ satisfying $h(a)\not=h(b)$. By Lemma~\ref{lem:induced_ultrafilters_are_points}, $h(\alpha(F))=F(h)$ for each $F\in 2^{2^A}$. By definition of $\pt(A)$, there is $p\in\pt(A)$ such that $e_A(p)=h$. Then 
\[u(a)(p) = e_A(p)(a)= h(a) \not= h(b) = e_A(p)(b)= u(b)(p),\]
hence $u(a)\not= u(b)$.

Next we show that $u$ is surjective. First consider the case when $A$ is an object of $\ODQPol$. Fix $f\in 2^{\pt(A)}$. Since $e_A\colon \pt(A) \to 2^A$ is a continuous embedding, the continuity of $f\colon \pt(A)\to 2$ and the compactness of $\pt(A)$ imply there is a sequence $\langle C_0, D_0\rangle, \ldots, \langle C_n, D_n\rangle$ of finite sets $C_i,D_i\subseteq A$ such that
\[f(p)=  \bigvee_{0\leq i \leq n}\left( \bigwedge_{c\in C_i} e_A(p)(c) \wedge \bigwedge_{c\in D_i} \neg e_A(p)(c) \right)\]
for each $p\in\pt(A)$. Define $F\colon 2^A \to 2$ as 
\[F(g) = \bigvee_{0\leq i \leq n}\left( \bigwedge_{c\in C_i} g(c) \wedge \bigwedge_{c\in D_i} \neg g(c) \right)\]
for each $g\in 2^A$. Then
\[ u(\alpha(F)) = \lambda p:\pt(A). F(e_A(p)) = \lambda p:\pt(A).f(p) = f, \]
hence $u$ is surjective.

Now consider the case when $A$ is an object of $\ZCHQPol$. Fix $f\in 2^{\pt(A)}$. Define $F\colon 2^A\to 2$ as 
\[F(g)=\left\{\begin{array}{ll}f(p)&\text{ if $p\in\pt(A)$ and $g=e_A(p)$}\\0 &\text{ if $g$ is not in the range of $e_A$}\end{array}\right.\]
for $g\in 2^A$. There is at most one $p\in\pt(A)$ with $g=e_A(p)$ because $e_A$ is an equalizer, so $F$ is well-defined. Also, $F$ is continuous because $2^A$ is discrete. Therefore, $F\in 2^{2^A}$ and
\begin{eqnarray*}
u(\alpha(F)) &=& \lambda p:\pt(A).F(e_A(p))\\
&=& \lambda p:\pt(A).f(p)\\
&=&f,
\end{eqnarray*}
hence $u$ is surjective.

It follows from Corollary~\ref{cor:bij_iso} and Corollary~\ref{cor:bij_iso_ch} that $u$ is a computable isomorphism. Let $v\colon 2^\pt(A) \to A$ be the inverse of $u$. It only remains to show that $v$ is a homomorphism of algebras. Using the fact that $u$ is an algebra homomorphism, we have
\begin{eqnarray*}
\alpha \circ 2^{2^v} &=& v\circ u\circ \alpha \circ 2^{2^v}\\
 &=&  v\circ 2^{\eta_{\pt(A)}}\circ 2^{2^u} \circ 2^{2^v}\\
&=&  v\circ 2^{\eta_{\pt(A)}}\circ 2^{2^{(u\circ v)}}\\
&=&  v\circ 2^{\eta_{\pt(A)}}.
\end{eqnarray*}
Therefore, $v$ is an algebra homomorphism.
\qed
\end{proof}

\subsection{The duality theorems}

Since the proof is the same for both $\ODQPol$ and $\ZCHQPol$, we explain it once using $\calC$ to denote the category $\ODQPol$ (or $\ZCHQPol$), and $\calC^*$ to denote the category $\ZCHQPol$ (or $\ODQPol$).

Define $\pt\colon \Alg(\calC) \to \calC^*$ as follows. Given an object $(A,\alpha)$ of  $\Alg(\calC)$, $\pt(A)$ (the \emph{space of points of $A$}) is defined as the equalizer of the following diagram:
\begin{center}
\begin{tikzcd}
\pt(A)\arrow[r, "e_A"]&2^A \arrow[r, yshift=0.7ex, "2^\alpha"]\arrow[r, yshift=-0.7ex, swap, "\eta_{(2^A)}"]& 2^{2^{2^A}}
\end{tikzcd}
\end{center}
To have a uniform definition of $\pt(A)$ and $e_A$, we assume they are defined by applying Theorem~\ref{thrm:comp_equalizers} to the computable maps $(A,\alpha)\mapsto 2^\alpha$ and $(A,\alpha)\mapsto \eta_{(2^A)}$. In particular, we assume that the PER defining $\pt(A)$ is the restriction of the PER defining $2^A$.

Given a homomorphism $h\colon (A,\alpha)\to (B,\beta)$, the two squares on the right of the diagram below commute because $h$ is a homomorphism and $\eta$ is a natural transformation, so there is a unique morphism $\pt(h)$ making the box on the left commute.
\begin{center}
\begin{tikzcd}
\pt(A)\arrow[r, "e_A"]& 2^A \arrow[r, yshift=0.7ex, "2^\alpha"]\arrow[r, yshift=-0.7ex, swap, "\eta_{(2^A)}"]& 2^{2^{2^A}}\\
\pt(B)\arrow[u,dashrightarrow,"\pt(h)"]\arrow[r, "e_B"]& 2^B\arrow[u, "2^h"] \arrow[r, yshift=0.7ex, "2^\beta"]\arrow[r, yshift=-0.7ex, swap, "\eta_{(2^B)}"]& 2^{2^{2^B}}\arrow[u,swap, "2^{2^{2^h}}"] 
\end{tikzcd}
\end{center}
Note that $\pt(h)$ is computable from $h$, because the graph of $\pt(h)$ is simply the restriction of the graph of $2^h$. Therefore, $\pt\colon \Alg(\calC) \to \calC^*$ is a computable contravariant functor.

The functor $2^{(-)}\colon \calC^*\to\Alg(\calC)$ is defined as mapping $X$ to the algebra $(2^X,2^{\eta_X})$, and mapping $f\colon X\to Y$ to the algebra homomorphism $2^f\colon 2^Y\to 2^X$. This is a computable contravariant functor by Lemma~\ref{lem:compcontfunctors}.

For each $X$ in $\calC^*$, if we apply $2^{(-)}$ followed by $\pt$, then because $X$ is sober there is a unique $\theta_X\colon \pt(2^X)\to X$ making the following diagram commute:
\begin{center}
\begin{tikzcd}
\pt(2^X)\arrow[d,swap,dashrightarrow,"\theta_X"]\arrow[rd,"e_{(2^X)}"]&&\\
X\arrow[r, "\eta_X"]& 2^{2^{X}}\arrow[r, yshift=0.7ex, "2^{2^{\eta_X}}"]\arrow[r, yshift=-0.7ex, swap, "\eta_{(2^{2^X})}"]& 2^{2^{2^{2^{X}}}}
\end{tikzcd}
\end{center}
Note that $\theta_X$ is an isomorphism because $\pt(2^X)$ and $X$ are both equalizers of the same diagram. Since the PER defining $\pt(2^X)$ is the restriction of the PER defining $2^{2^X}$, the graph of $\theta_X$ is the inverse of the graph of $\eta_X$. This determines a computable natural isomorphism $\theta\colon \pt\circ 2^{(-)} \to 1_{\calC^*}$.

For each $(A,\alpha)$ in $\Alg(\calC)$, if we apply $\pt$ followed by $2^{(-)}$, then because $(A,\alpha)$ is spatial there is a unique algebra isomorphism $\psi_A\colon A \to \pt(2^A)$ making the following diagram commute:
\begin{center}
\begin{tikzcd}
2^{\pt(A)}&\arrow[l,swap, "2^{(e_A)}"]\arrow[dl,"\alpha"] 2^{2^{A}}& \arrow[l, yshift=0.7ex, swap,"2^{2^\alpha}"]\arrow[l, yshift=-0.7ex, "2^{\eta_{(2^A)}}"]2^{2^{2^{2^A}}}\\
A\arrow[u,dashrightarrow,"\psi_A"]&&
\end{tikzcd}
\end{center}
It was shown in the proof of Theorem~\ref{thm:spatial} that $\psi_A = 2^{e_A}\circ \eta_A$, hence it is computable from $(A,\alpha)$. This determines a computable natural isomorphism $\psi\colon 1_{\Alg(\calC)}\to 2^{(-)}\circ \pt$.

We summarize the dualities with the following diagram and theorem, where $U$ denotes the forgetful functors (functors in the horizontal direction are contravariant, and functors in the vertical direction are covariant).

\begin{center}
\begin{tikzcd}
\ODQPol\arrow[r,swap,yshift=-0.7ex, "2^{(-)}"]\arrow[d,swap,xshift=-1ex,"2^{2^{(-)}}"]\arrow[d,phantom,swap,"\dashv"]& \Alg(\ZCHQPol)\arrow[l,swap,yshift=0.7ex, "\pt"]\arrow[d,xshift=1ex,"U"]\arrow[d,phantom,swap,"\dashv"]\\
\Alg(\ODQPol)\arrow[u,swap,xshift=1ex,"U"]\arrow[r,yshift=0.7ex, "\pt"]& \ZCHQPol\arrow[l,yshift=-0.7ex, "2^{(-)}"]\arrow[u,xshift=-1ex,"2^{2^{(-)}}"]
\end{tikzcd}
\end{center}

\begin{theorem}[Computable Stone dualities]
\noindent

1. The effective quasi-Polish categories $\Alg(\ODQPol)$ and $\ZCHQPol$ are computably dually equivalent, via the computable contravariant functors $\pt\colon \Alg(\ODQPol)\to \ZCHQPol$ and $2^{(-)}\colon \ZCHQPol\to  \Alg(\ODQPol)$, and the computable natural isomorphisms $\theta\colon \pt\circ 2^{(-)} \to 1_{\ZCHQPol}$ and $\psi \colon 1_{\Alg(\ODQPol)}\to 2^{(-)}\circ \pt$.

2. The effective quasi-Polish categories $\Alg(\ZCHQPol)$ and $\ODQPol$ are computably dually equivalent, via the computable contravariant functors $\pt\colon \Alg(\ZCHQPol)\to \ODQPol$ and $2^{(-)}\colon \ODQPol\to  \Alg(\ZCHQPol)$, and the computable natural isomorphisms $\theta\colon \pt\circ 2^{(-)} \to 1_{\ODQPol}$ and $\psi \colon 1_{\Alg(\ZCHQPol)}\to 2^{(-)}\circ \pt$.
\qed
\end{theorem}

\bibliographystyle{plainurl}
\bibliography{myrefs}
\end{document}